\documentclass[11pt]{amsart}
\allowdisplaybreaks[4]

\usepackage{amsmath, amsthm, amscd, amssymb,amsfonts}
\usepackage{color}
\usepackage{mathrsfs}
\usepackage[all]{xy}
\usepackage{hyperref}
\usepackage{enumitem}

\usepackage{varwidth}

\numberwithin{equation}{section}

\newtheorem{thm}{Theorem}[section]

\newtheorem{cor}[thm]{Corollary}
\newtheorem*{cor*}{Corollary}
\newtheorem{lem}[thm]{Lemma}

\newtheorem*{con*}{Conjecture}
\newtheorem*{prob*}{Problem}

\theoremstyle{definition}
\newtheorem{defn}[thm]{Definition}

\theoremstyle{remark}
\newtheorem{rem}[thm]{Remark}

\newcommand{\ds}{\displaystyle}

\usepackage{cite}

\begin{document}
\title{Topological K-Homology for Non-proper Actions and Assembly Maps}

\author[Kun Wang]{Kun Wang}
\address{Beijing Institute of Technology, Zhuhai Campus, 
Zhuhai, Guangdong, China}
\email{20381@bitzh.edu.cn}

\begin{abstract}
For a countable discrete group $G$, we construct a new and concrete model for the equivariant topological $K$-homology theory of $G$, which is defined for \textit{all} $G$-actions, not just for proper $G$-actions. 
The construction of our model combines some of the ideas from the localization algebra approach to the coarse Baum-Connes conjecture and the controlled algebra approach to the Farrell-Jones conjecture. It brings new ideas into the study of the Baum-Connes conjecture. First, as the model is defined for all $G$-actions, we are able to define relative assembly maps.  We prove, by using an induction structure of our theory, a transitivity principle that can be used to verify when a relative assembly map is an isomorphism. This promotes the study of the original assembly map relative to the family of finite subgroups to assembly maps relative to any family of subgroups of a group.  Second, our new model is suitable for the use of the method of controlled topology and controlled algebra, which plays an important role in proving the coarse Baum-Connes conjecture and the Farrell-Jones conjecture for a large class of groups.

\end{abstract}

\maketitle
\section{Introduction}

For a discrete group $G$, the corresponding topological $K$-groups $K_*(C_r^*(G))$, algebraic $K$- and $L$-groups $K_*(\mathbb Z[G]), L_*(\mathbb Z[G])$ are important objects of study in topology, geometry and analysis, where $C_r^*(G)$ and $\mathbb Z[G]$ are the reduced group $C^*$-algebra and the integral group ring of $G$ respectively. The structures of these groups are predicted by the Baum-Connes conjecture and the Farrell-Jones conjecture respectively. The assembly maps which formulate these two conjectures have deep meanings in topology, geometry and analysis, so that they have many important implications in these areas. For example, they imply the Borel conjecture on topological rigidity of closed aspherical manifolds and the Novikov conjecture on homotopy invariance of higher signatures,  see \cite{LR} for a detailed account about these two conjectures and their applications.

The formulations of these conjectures rely on certain equivariant homology theories. For Baum-Connes, it is equivariant topological $K$-homology theory and for Farrell-Jones, they are equivariant algebraic $K$- and $L$-homology theories. In \cite{DL}, Davis and L{\"u}ck give a unified, but abstract, approach to the constructions of these equivariant homology theories using homotopy theoretic methods.  The theories so constructed are defined for \textit{all} $G$-actions, not just for proper $G$-actions. However, this abstract homotopy theoretic approach is not convenient for the actual proof of the conjectures. In \cite{BFJR}\cite{BR}, the authors construct concrete models for equivariant algebraic $K$- and $L$-homology theories using controlled algebra method. These concrete models are defined for \textit{all} $G$-actions as well and have been used for proving the Farrell-Jones conjecture for a large class of groups \cite{BLR2}\cite{BL2}\cite{BLRR}\cite{BFL}\cite{BB}\cite{FW1}\cite{WC1}\cite{WC2}. The fact that these models are defined for \textit{all} $G$-actions plays a crucial role for the proofs, as it enables one to define \textit{relative assembly maps} and therefore promote the study of the original assembly maps to assembly maps relative to any family of subgroups of $G$. Now for an arbitrary family of subgroups of $G$, geometric conditions have been formulated in \cite{BLR2}\cite{BL2} so that the corresponding assembly maps can be verified to be isomorphisms through their concrete models and therefore prove the Farrell-Jones conjecture for $G$ via relative assembly maps.


However, the concrete models \cite{KAS1}\cite{BCH}\cite{BHS}\cite{WY} constructed so far for equivariant topological $K$-homology theory are only defined for proper group actions, not for all group actions, and this  obstructs the construction of relative assembly maps. Therefore, in order to carry the general strategy used to prove the Farrell-Jones conjecture over to the study of the Baum-Connes conjecture, it is desirable to have a concrete model for equivariant topological $K$-homology theory which is defined for \textit{all} group actions, and the construction of such a model is the main purpose of the paper (the construction of such a concrete model is certainly of its own interest as well). Our construction of such a model is very concrete and is  suitable for the use of the method of controlled topology and controlled algebra, which plays an important role in proving the coarse Baum-Connes conjecture and the Farrell-Jones conjecture for some important classes of groups\cite{YGL3}\cite{BLR2}\cite{BL2}.

The paper is organized as follows. In section \ref{basic_constructions_and_main_results}, we outline the basic constructions, define the groups $LK^G_*(X)$ for any $G$-space $X$ and state the main result (Thereom \ref{eqkhomology}), that $LK^G_*$ is a model for the equivariant topological $K$-homology theory of $G$. We also define \textit{relative assembly maps} (Definition \ref{relative_assembly_map}) in the context of our $LK^G_*$-theory and state the \textit{transitivity principle} (Theorem \ref{transitivity_principle}). In section \ref{lemmas}, we prove various lemmas that will be needed for the proof of the main result, and then in section \ref{proof_of_the_main_theorem}, we prove Theorem \ref{eqkhomology}. In section \ref{induction}, we show, for any subgroup $H<G$, there is an \textit{induction structure} (Theorem \ref{induction_str}) from $LK^H_*$ to $LK^G_*$. Finally in section \ref{transitivity}, by using the induction structure, we prove the transitivity principle. We remark that our construction will also work for equivariant topological $K$-homology theory with coefficients, but in oder to keep the treatment a bit cleaner, we do not consider it here.

\section{Basic Constructions and Main Results}\label{basic_constructions_and_main_results}

In this section, we outline the basic constructions and state the main results. All claims in this section are stated without proofs and will be justified in later sections.

\subsection{The algebra $C^*(G\times X)$}
    Let $\mathcal H$ be a separable infinite dimensional complex Hilbert space, and $X$ be a set. Denote by $\mathcal H_X$  the direct sum $\bigoplus_X\mathcal H$, i.e.
    $$\mathcal H_X: =\{v: X\rightarrow\mathcal H\ |\ \sum_{x\in X} \parallel v(x)\parallel^2<\infty\}$$
where $\sum_{x\in X} \parallel v(x)\parallel^2<\infty$ means the net $\{a_F=\sum_{x\in F}\parallel v(x)\parallel^2\ :\ F\subseteq X, F\ \text{is finite}\}$ converges to a finite number. $\mathcal H_X$ is a Hilbert space in the obvious way.  If $A\subseteq X$ is a subset, then $\mathcal H_A$ is naturally identified with a subspace of $\mathcal H_X$. If $x\in X$, then we denote $\mathcal H_{\{x\}}$ by $\mathcal H_x$.

    Let $X, Y$ be two sets. For any bounded linear operator  $T: \mathcal H_X\rightarrow\mathcal H_Y$, and any $x\in X, y\in Y$, its\textit{ $(y, x)$-component}, denoted by $T_{y, x}$,  is defined to be the composition $\pi_{y}\circ T\circ\iota_x: \mathcal H\rightarrow\mathcal H$, where $\iota_x:\mathcal H\rightarrow \mathcal H_X$ is the inclusion of $\mathcal H$ into the $x$-th summand of $\mathcal H_X$, and $\pi_{y}:  \mathcal H_Y\rightarrow \mathcal H$ is the orthogonal projection of $\mathcal H_Y$ onto its $y$-th summand. In this way, we can think of $T$ as a matrix $(T_{y, x})$ indexed by $Y\times X$, with each entry a bounded operator on $\mathcal H$.

    The \textit{support} of $T$,  denoted by supp$(T)$, is defined to be the following subset of $Y\times X$:
    $$\text{supp}(T):=\{(y, x)\in Y\times X\ |\ T_{y, x}\neq 0\}.$$
$T$ is called \textit{locally compact} if $T_{y, x}$ is a compact operator on $\mathcal H$ for any $y\in Y, x\in X$. $T$ is called \textit{row finite} if for every $y\in Y$, there are only finitely many $x\in X$ so that $T_{y, x}\neq 0$. Similarly, we can define \textit{column finite} operators. If $\phi: X\longrightarrow Y$ is a map, then we call $T$ \textit{covers} $\phi$ if $\text{supp}(T)\subseteq\{(\phi(x), x)\in Y\times X\ |x\in X\}$.
    
     Let $G$ be a discrete group.  If $X$ is a left $G$-set, then there is a left $G$-action $U^X$ on $\mathcal H_X$ via
     $$(U^X_gv)(x):=v({g^{-1}x}),\ \ g\in G, v\in\mathcal H_X, x\in X.$$
     If $X$ and $Y$ are two left $G$-sets, then there is a left $G$-action 
     $L$ on $B(\mathcal H_X, \mathcal H_Y)$, the space of all bounded linear operators from $\mathcal H_X$ to $\mathcal H_Y$, via
     $$L_gT=U^Y_g\circ T\circ U^X_{g^{-1}}.$$
     In terms of components, we have 
     $$(L_gT)_{y, x}=T_{g^{-1}y,\ g^{-1}x}.$$
     $T$ is called \textit{$G$-invariant} if $L_gT=T$ for all $g\in G$. In terms of components, this means 
     $$T_{gy,\ gx}=T_{y, x}$$ for all $g\in G, x\in X, y\in Y$.

\begin{defn}\label{Roealg}
	 Let $G$ be a discrete group and $X$ be a left $G$-set. Endow $G\times X$ with the diagonal  $G$-action ($G$ acts on $G$ by left translations).  Define $C[G\times X]$ to be the set of all bounded linear  operators $T:\mathcal H_{G\times X}\rightarrow\mathcal H_{G\times X}$ with the following properties: 
\begin{enumerate}
	\item $T$ is $G$-invariant;
	\item $T$ is locally compact;
    \item $T$ is $G$-finite: there exists a finite subset $F\subseteq X$ so that supp$(T)\subseteq G_F\times G_F$, where $G_F:=\{(g, gx)\in G\times X\ |\ g\in G, x\in F\}$;
    \item $T$ has finite propagation in the $G$-direction: there exists a finite subset $S\subseteq G$,  so that if $T_{(g', x'), (g, x)}\neq 0$, then $g^{-1}g'\in S$.
\end{enumerate}
When $X$ is the one-point space, we denote $C[G\times X]$ by $C[G]$.
\end{defn}

    It is easy to see that elements of $C[G\times X]$ are row and column finite.  We shall see in Lemma \ref{lemstr1} that $C[G\times X]$ is a $*$-subalgebra of $B(\mathcal H_{G\times X})$, the $C^*$-algebra of all bounded linear operators on $\mathcal H_{G\times X}$. 
    
\begin{defn}\label{Completed_Roealg} Define $C^*(G\times X)$ to be the norm closure of $C[G\times X]$ in $B(\mathcal H_{G\times X})$. When $X$ is the one-point space, we denote $C^*(G\times X)$ by $C^*(G)$.
\end{defn}   

    We will see in Lemma \ref{lemstr2} that $C^*(G\times X)$ is stably equivalent to $C^*_r(G)$. In particular, they have the same $K$-theory. Therefore the $K$-theory of  $C^*(G\times X)$ is independent of the $G$-space $X$.

\subsection{The localization algebra $C^*_L(G\times X)$}
    Let $X$ be a topological space with a left $G$-action by homeomorphisms. We now use Yu's localization algebra technique \cite{YGL1}\cite{YGL3} to construct $C^*$-algebras that depend on the topology of $X$. However, in order to make the definition work for all topological spaces, not just for metric spaces, we replace the metric control condition in Yu's definition by the \textit{equivariant continuous control condition} which is used in the controlled algebra approach to the Farrell-Jones conjecture \cite{BFJR}. Here, we have to slightly modify it so that it adapts to the localization algebra technique.

\begin{defn}\label{ecc}
	Let $G$ be a discrete group and $X$ be a topological space with a left $G$-action by homeomorphisms. A subset $E\subseteq X\times X\times [0, \infty)$ is called \textit{equivariantly continuously controlled at infinity} if the following holds:
	\begin{enumerate}
		\item for every $x\in X$ and every $G_x$-invariant open neighborhood $U$ of $x$ in $X$, there exist $N>0$ and a $G_x$-invariant open neighborhood $V\subseteq U$ of $x$, so that if $t>N$ and $(x_1, x_2, t)\in E$, then one of $x_1,x_2$ lies in $V$ implies the other lies in $U$;
		\item $E$ is symmetric: if $(x_1, x_2, t)\in E$, then $(x_2, x_1, t)\in E$;
		\item $E$ is $G$-invariant: if $(x_1, x_2, t)\in E$, then $(gx_1, gx_2, t)\in E$ for all $g\in G$;
		\item $\Delta\subseteq E$, where $\Delta=\{(x, x, t)\ |\ x\in X, \ t\ge 0\}$ is the diagonal.
    \end{enumerate}
    Let $\mathscr E^G_X$ denote the set of all subsets of $X\times X\times[0, \infty)$ that are equivariantly continuously controlled at infinity.
\end{defn}
    
    Intuitively, if $E\in\mathscr E^G_X$, then $(x_1, x_2, t)\in E$ implies $x_1, x_2$ are very close when $t$ is very large. 
    

\begin{defn}\label{localalg}
	Define $C_L[G\times X]$ to be the set of all bounded, uniformly continuous functions 
    $$f: [0, \infty)\rightarrow C[G\times X]$$ 
    with the property that  
    there exist a finite subset $S\subseteq G$, a compact subset $K\subseteq X$, a $G$-invariant countable subset $\Lambda\subseteq X$, and $E\in\mathscr E^G_X$ so that $f$ is \textit{$(S, K, \Lambda, E)$-controlled}: for any $t\in[0, \infty)$,  if $f(t)_{(g', x'), (g, x)}\neq 0$, then $g^{-1}g'\in S$, $x, x'\in\Lambda$,  $g^{-1}x, (g')^{-1}x'\in K$ and  $(x', x, t)\in E$. 
\end{defn}

    An element of $C_L[G\times X]$ may be thought as a family of $G\times X$ by $G\times X$ matrices, whose entries are compact operators on $\mathcal H$, parametrized by $t\in[0,\infty)$, and this family of matrices has to satisfy some additional control conditions. Lemma \ref{localalglem} in the next section shows that $C_L[G\times X]$ is a $*$-subalgebra of $\mathcal{BU}\big([0, \infty), C^*(G\times X)\big)$,  the $C^*$-algebra of all bounded and  uniformly continuous functions from $[0, \infty)$ to $C^*(G\times X)$ with the sup norm $$\parallel f\parallel:=\sup_{t\ge 0}\parallel f(t)\parallel,\ \ f\in \mathcal{BU}\big([0, \infty), C^*(G\times X)\big).$$


\begin{defn}\label{localizationalgebra}
	The \textit{localization algebra} (with $G$-compact support) of a $G$-space $X$, denoted by $C^*_L(G\times X)$, is defined to be the $C^*$-algebra which is the norm closure of $C_L[G\times X]$ in $\mathcal{BU}\big([0, \infty), C^*(G\times X)\big)$. 
\end{defn}



\subsection{The equivariant $K$-homology theory $LK^G_*$}    

To state the main theorem, we will further need the following definition:
    
\begin{defn}\label{local_alg_subspace}
    Let $A$ be a $G$-invariant subspace of $X$. We define $C_L[G\times X; A]$ to be the subset of $C_L[G\times X]$ consisting of all $f$ so that $f(t)_{(g', x'), (g, x)}\neq 0$ implies $x, x'\in A$. We also define $C_L[G\times X; \langle A\rangle]$ to be the subset of $C_L[G\times X]$ consisitng of all $f$ so that there exist $E\in \mathscr E^G_X$ and compact $K\subseteq X$ with the following property: if $f(t)_{(g', x'), (g, x)}\neq 0$, then $g^{-1}x, (g')^{-1}x'\in K$ and there exist $a, a'\in A$ so that $g^{-1}a, (g')^{-1}a'\in K$ and $(a, x, t), (a', x', t)\in E$. Let $C^*_L(G\times X; A)$ and $C^*_L(G\times X; \langle A\rangle)$ be their closures in $C^*_L(G\times X)$.    
\end{defn}    
     
     By Lemma \ref{ideal} in the next section, $C^*_L(G\times X; A)$ is a $C^*$-subalgebra of $C^*_L(G\times X)$ and $C^*_L(G\times X; \langle A\rangle)$ is an ideal of $C^*_L(G\times X)$. $C^*_L(G\times X; A)\subseteq C^*_L(G\times X; \langle A\rangle)$ and they have isomorphic $K$-groups.

\begin{defn}\label{relative}
	Let $A$ be a $G$-invariant subspace of $X$. For $i=0, 1$, we define 
	$$LK^G_i(X, A)=K_i\big(C^*_L(G\times X)/C^*_L(G\times X; \langle A\rangle)\big).$$ In particular, $LK^G_i(X)=LK^G_i(X, \emptyset)=K_i\big(C^*_L(G\times X)\big)$.
\end{defn}

    We are now ready to state the main theorem:

\begin{thm}\label{eqkhomology}
	Let $G$ be a countable discrete group. Then $LK^G_*$  is a $G$-equivariant homology theory (with $G$-compact support) on the category of $G$-$CW$-pairs. More precisely, we have:
\begin{enumerate} 
	\item Functoriality:
	     every $G$-equivariant continuous map $\phi: (X,A)\rightarrow (Y, B)$ between $G$-$CW$-pairs induces a homomorphism
         $$\phi_*: LK^G_*(X, A)\longrightarrow LK^G_*(Y, B) $$
         so that $id_*=id$ and $(\phi\circ\psi)_*=\phi_*\circ\psi_*$;
    \item Exactness:
         for every $G$-$CW$-pair $(X, A)$, there is a six-term cyclic exact sequence 
         $$\xymatrix{
         	\cdots\ar[r]& LK^G_*(A)\ar[r]^{j_*}& LK^G_*(X)\ar[r]^{k_*\ }& LK^G_*(X, A)\ar[r]^{\partial}& LK^G_{*-1}(A)\ar[r] &\cdots}$$
         which is natural, where $j: A\rightarrow X, k: X\rightarrow (X, A)$ are inclusions;
    \item $G$-homotopy invariance:
    if two $G$-equivariant continuous maps $\phi, \psi: (X, A)\rightarrow (Y, B)$ are $G$-equivariantly homotopic, then $\phi_*=\psi_*$;
    \item Excision:
         for every $G$-$CW$-triad $(X, A, B)$ with $X=A\cup B$, 
         the inclusion $j: (A, A\cap B)\rightarrow(X, B)$ induces an isomorphism
         $$j_*: LK^G_*(A, A\cap B)\tilde\longrightarrow LK^G_*(X, B);$$
    \item Additivity:
         if $X=\coprod_{\alpha\in I} X_\alpha$, a disjoint union of a family of $G$-CW-complexes, then the inclusions $j_\alpha: X_\alpha\longrightarrow X$ induce an isomorphism
         $$\bigoplus_{\alpha\in I}j_\alpha: \bigoplus_{\alpha\in I}LK_*^G(X_\alpha)\tilde\longrightarrow LK^G_*(X).$$   
    Moreover, $LK^G_*$ satisfies:
    
    \vskip 5pt
    \item Value at $G/H$: $LK^G_*(G/H)\cong K_*\big(C^*_r(H)\big)$ for every subgroup $H<G$. In particular, 
    $LK^G_*(\{pt\})\cong K_*\big(C^*_r(G)\big)$.
\end{enumerate}
\end{thm}

\begin{rem}\label{pushout}
	The excision of $LK^G_*$ is equivalent to the following: let 
	$(X, A)$ be a $G$-CW pair and $\phi: A\rightarrow Y$ be a cellular $G$-map.
	Let $X\cup_\phi Y$ be the $G$-CW complex obtained from $Y$ by attaching $X$ via $\phi$.  Then the canonical map $\Phi: (X, A)\longrightarrow (X\cup_\phi Y, Y)$ induces an isomorphism
	$$\Phi_*: LK^G_*(X, A)\tilde{\longrightarrow} LK^G_*(X\cup_\phi Y, Y).$$
	Clearly, the excision is a special case of this. For the converse, it suffices to show the collapsing map $q: (X, A)\longrightarrow (X/A, A/A)$ induces an isomorphism $q_*: LK^G_*(X, A)\tilde{\longrightarrow} LK^G_*(X/A, A/A)$, since the induced map $\bar\Phi: (X/A, A/A)\longrightarrow \big((X\cup_\phi Y)/Y, Y/Y\big)$ is a $G$-homeomorphism. Now $q_*$ is an isomorphism follows from the consideration of the following diagram
	$$
	\xymatrix{ 
			(X, A)\ar[r]\ar[d] & (X\cup\mathcal CA, \mathcal CA)\ar[d]\\ 
		    (X/A, A/A)\ar[r] & \big((X\cup\mathcal CA)/\mathcal CA, \mathcal CA/\mathcal CA\big)
		}
	$$ 
	where $\mathcal CA=A\times[0, 1]/A\times\{0\}$ is the cone over $A$. The top horizontal map induces an isomorphism on $LK^G_*$ by excision. The bottom horizontal map induces an isomorphism on $LK^G_*$ since it is a $G$-homeomorphism. The right vertical map induces an isomorphism on $LK^G_*$ since it is a $G$-homotopy equivalence (because $\mathcal CA$ is $G$-equivariantly contractible and the inclusion of $\mathcal CA$ in $X\cup\mathcal CA$ is a $G$-cofibration). It follows that $q_*$ is an isomorphism. All these arguments are routine generalizations from the  non-equivariant case to the equivariant case.
\end{rem}

\begin{rem}\label{mayer-vietoris}
	Let $\Phi: (X, A)\longrightarrow (X\cup_\phi Y, Y)$ be as in Remark \ref{pushout}. Then, as a formal consequence,  we have a Mayer-Vietoris exact sequence 
	$$
	\xymatrix{ 
		\cdots\ar[r]& LK^G_*(A)\ar[r] & LK^G_*(X)\oplus LK^G_*(Y)\ar[r] & LK^G_*(X\cup_\phi Y)\ar[r] &\cdots\\ 
	}
	$$ 	
\end{rem}

\begin{rem}\label{directed union}
	The additivity of $LK^G_*$ is equivalent to the following directed union property of $LK^G_*$: if a $G$-CW complex $X$ is the directed union of a family of its $G$-CW subcomplexes $\{X_\alpha: \alpha\in I\}$ directed by inclusions, then the canonical map $\text{colim}_{\alpha\in I} LK^G_*(X_\alpha)\tilde{\longrightarrow} LK^G_*(X)$ is an isomorphism.
	This is also a routine generalization from the  non-equivariant case to the equivariant case. In our case, the directed union property of $LK^G_*$ follows easily from the definition of $C^*_L(G\times X)$ and we will use it to obtain the additivity. As a special case, $LK^G_*(X)$ is the direct colimit of $LK^G_*(X_\alpha)$, where $X_\alpha$ ranges over all $G$-compact subcomplexes of $X$. 
\end{rem}


\subsection{Relative assembly map and the transitivity principle}
As $LK^G_*$ is defined for all $G$-CW-complexes, not just for proper $G$-CW-complexes, we can define \textit{relative assembly maps} and use it to study the Baum-Connes conjecture.

We shall first need to recall the notion of \textit{classifying space} relative to a \textit{family of subgroups} of a group. Let $\mathcal F$ be a set of subgroups of $G$. It is called a \textit{family of subgroups} of $G$ if it is closed under taking subgroups and conjugations. A model for the \textit{classifying space} of $G$ relative to $\mathcal F$, denoted by $E_\mathcal FG$, is a $G$-CW complex with the property that all of its isotropy groups belong to $\mathcal F$, and for any $G$-CW-complex $X$ with stabilizers in $\mathcal F$, there is a $G$-equivariant continuous map from $X$ to $E_\mathcal FG$ which is unique up to $G$-homotopy. $E_\mathcal FG$ is unique up to $G$-homotopy equivalence, and is also characterized by the property that all of its isotropy groups belong to $\mathcal F$, and for any $H\in\mathcal F$, its fixed point set is weakly contractible. It is customary to denote $E_{\{1\}}G$ by $EG$, $E_{\mathcal{FIN}}G$ by $\underline{E}G$ and $E_{\mathcal{VC}}G$ by $\underline{\underline{E}}G$, where $\{1\}, \mathcal{FIN}, \mathcal{VC}$ are the families of the trivial subgroup, finite subgroups, and virtually cyclic subgroups of $G$ respectively. If $\mathcal F=\mathcal{ALL}$ is the family of all subgroups of $G$, then we always take $E_{\mathcal{ALL}}G=\{pt\}$, the one point space with the trivial $G$-action. More information about classifying spaces for families of subgroups can be founded in \cite{Luck2}.

\begin{rem} The space $\underline{E}G$ defined above is Topologist's classifying space for proper $G$-
    actions. However, there is also the Analyst's classifying space for proper $G$-actions as defined and  
    constructed in \cite{BCH}. The standard models for them are different in general, as Topologists' is 
    usually not metrizable while Analysts' is always metrizable. This ambiguity can be eliminated as one can 
    show they are always $G$-homotopy equivalent, see the Appendix of \cite{Valette} for a proof. 
\end{rem}

\begin{defn}\label{relative_assembly_map} Let $\mathcal F$ and $\mathcal F'$ be two families of subgroups of  
    $G$ with $\mathcal F\subseteq\mathcal F'$. Then the unique (up to $G$-homotopy) $G$-map  $E_{\mathcal{F}} 
    G\rightarrow E_{\mathcal F'}G$ induces a homomorphism 
    $$A^{\mathcal{F}'}_{{\mathcal F}}: LK^G_*(E_{\mathcal{F}}G)\longrightarrow LK^G_*(E_{\mathcal{F}'}G).$$
    It is called the \textit{relative assembly map} from $\mathcal F$ to $\mathcal F'$. When $\mathcal F'=  
    \mathcal{ALL}$, the family of all subgroups of $G$, the relative assembly map will be denoted by  
    $A_{\mathcal F}$. In this case, since $E_{\mathcal{ALL}}G=\{pt\}$ and $LK^G_*(\{pt\})=K_*(C^*_r(G))$, we 
    have
    $$A_{\mathcal F}: LK^G_*(E_{\mathcal{F}}G)\longrightarrow K_*(C^*_r(G)).$$
\end{defn}

\begin{defn}\label{baum_connes} Let $\mathcal F$ be a family of subgroups of $G$. Then the \textit{Baum- 
    Connes conjecture for $G$ relative to $\mathcal F$} is the statement that the assembly map
    $$A_{\mathcal{F}}: LK^G_*(E_{\mathcal{F}}G)\longrightarrow K_*(C^*_r(G))$$
    is an isomorphism. When we say the Baum-Connes conjecture for $G$ without mentioning the family of    
    subgroups, we always mean $\mathcal F=\mathcal{FIN}$. 
\end{defn}  
  
\begin{rem}\label{obs_loc_alg} For any $G$-space $X$, we have an evaluation map
	$$ev_X:\ C^*_L(G\times X)\longrightarrow C^*(G\times X)$$
    $$ev_X(f)=f(0)$$
    which is a surjective $C^*$-algebra homomorphism.
    The kernel of this map is denoted by $C^*_{L, 0}(G\times X)$. It can then be shown that the Baum-Connes conjecture for $G$ relative to $\mathcal F$ holds if and only if the $K$-theory of 
$C^*_{L, 0}(G\times E_{\mathcal F}G)$ is trivial.
\end{rem}

   The following \textit{transitivity principle} is a very useful criteria that can be used to verify when a relative assembly map is an isomorphism:

\begin{thm}\label{transitivity_principle}\textnormal{(Transitivity Principle)} Let $\mathcal F$ and $\mathcal    
    F'$ be two families of subgroups of $G$ with $\mathcal F\subseteq\mathcal F'$. If the Baum-Connes     
    conjecture for $H$ relative to $H\cap\mathcal F$ holds for any $H\in\mathcal F'$, then the relative 
    assembly map 
    $$A_{\mathcal{F}'}^{{\mathcal F}}: LK^G_*(E_{\mathcal{F}}G)\longrightarrow LK^G_*(E_{\mathcal{F}'}G)$$
    is an isomorphism. Therefore, if the Baum-Connes conjecture holds for $G$ relative to $\mathcal F'$, then   
    it also holds for $G$ relative to $\mathcal F$.
\end{thm}

\begin{rem}   
   Let $\mathcal F$ be a family of subgroups of $G$ which contains the family $\mathcal{FIN}$ of finite subgroups of $G$. According to the transitivity principle, to show the Baum-Connes conjecture for $G$, it suffices to show 1) the Baum-Connes conjecture for $H$ holds for any $H\in\mathcal F$ (which is usually much easier than that for $G$ as $G$ is bigger) and 2) the Baum-Connes conjecture for $G$ relative to $\mathcal F$ holds. This is also the general strategy used to prove the Farrell-Jones conjecture for a large class of groups. General geometric conditions have been formulated in \cite{BLR2}\cite{BL2} from which the Farrell-Jones conjecture for $G$ relative to $\mathcal F$ can be deduced. It is our hope that this strategy and these geometric conditions can also be used  to prove the Baum-Connes conjecture through our model.
\end{rem}

\section{Lemmas}\label{lemmas}

In this section, we prove various lemmas that will be needed for the proof of our main results. Throughout the section, $G$ is a countable discrete group and $\mathcal H$ is a separable and infinite dimensional complex Hilbert space. 

\subsection{The algebras $C[G\times X]$ and $C^*(G\times X)$} We will first have a closer look at the sets
$C[G\times X]$ and $C^*(G\times X)$  defined in Section \ref{basic_constructions_and_main_results}. We need the following basic facts:

\begin{lem}\label{T-properties} Let $X, Y, Z$ be three sets and $\mathcal H_X, \mathcal H_Y, \mathcal H_Z$ be the corresponding Hilbert spaces as defined in Section \ref{basic_constructions_and_main_results}. Let $T: \mathcal H_X\rightarrow\mathcal H_Y$ and $S:\mathcal H_Y\rightarrow\mathcal H_Z$ be two bounded linear operators, and $T^*:\mathcal H_Y\rightarrow\mathcal H_X$ be the adjoint of $T$. Then
	\begin{enumerate}
		\item $(T^*)_{x, y}=(T_{y, x})^*$ for any $x\in X, y\in Y$;
		\item for any $x\in X, z\in Z$ and $v\in\mathcal H$, we have
		$$(S\circ T)_{z, x}(v)=\displaystyle\sum_{y\in Y}(S_{z, y}\circ T_{y, x})(v),$$
		where the equality means the net $\{\ds\sum_{y\in F}(S_{z, y}\circ T_{y, x})(v)\ |\ F\subseteq Y\ \text{is\ finite}\}$ converges to $(S\circ T)_{z,x}(v)$.
		\item $\textnormal{supp}(T^*)=\textnormal{supp}(T)^{\textnormal{op}}$ and 
		$\textnormal{supp}(S\circ T)\subseteq\textnormal{supp}(S)\circ\textnormal{supp}(T)$, where
		$$\textnormal{supp}(T)^{\textnormal{op}}=\{(x, y)\in X\times Y\ |\ (y, x)\in\textnormal{supp}(T)\},$$
		$$\textnormal{supp}(S)\circ\textnormal{supp}(T)=\{(z, x)\in Z\times X\ |\ \exists y\in Y,\ \textnormal{s.t.}\ (z, y)\in\textnormal{supp}(S),\ (y, x)\in\textnormal{supp}(T) \}.$$
		\item Let $x\in X, z\in Z$. If there are only finitely many $y\in Y$ so that $S_{z, y}\circ T_{y, x}\ne 0$,  then $$(S\circ T)_{z, x}=\ds\sum_{y\in Y}S_{z, y}\circ T_{y, x}.$$
		\item If one of $S$ and $T$ is locally compact, then $S$ is row finite or $T$ is column finite implies $S\circ T$ is locally compact.
	\end{enumerate}
\end{lem}

\begin{proof}
\begin{enumerate}
\item For any $v, w\in\mathcal H$, we have
\begin{align*}
\langle v, (T_{y, x})^* w\rangle &=\langle T_{y, x}v, w\rangle\\
&=\langle(\pi_y\circ T\circ\iota_x)(v), w\rangle\\
&=\langle v, (\pi_y\circ T\circ\iota_x)^*(w)\rangle\\
&=\langle v, (\iota_x^*\circ T^*\circ\pi_y^*)(w)\rangle\\
&=\langle v, (\pi_x\circ T^*\circ\iota_y)(w)\rangle\\
&=\langle v, (T^*)_{x, y}(w)\rangle.
\end{align*}
Therefore $(T^*)_{x, y}=(T_{y, x})^*$.
\item  We compute 
\begin{align*}
(S\circ T)_{z, x}(v)&=(\pi_z\circ S\circ T\circ\iota_x)(v)\\
&=\Big(\pi_z\circ S\Big)\Big(\sum_{y\in Y}\iota_y\big(\pi_y(T(\iota_x(v)))\big)\Big)\\
&=\sum_{y\in Y}\big(\pi_z\circ S\circ\iota_y\circ\pi_y\circ T\circ\iota_x\big)(v)\\
&=\sum_{y\in Y}(S_{z, y}\circ T_{y, x})(v)
\end{align*}
\end{enumerate}
Parts (3)-(5) of the lemma follow easily from parts (1) and (2).
\end{proof}

    Let $\mathcal K$ denote the algebra of all compact operators on $\mathcal H$, and $\rho$ denote the \textit{right regular representation} of the group algebra $\mathbb C[G]$ on $l^2(G)$. Therefore 
   $\rho(g)(\delta_h)=\delta_{hg}$ for any $g, h\in G$, where $\delta_h\in l^2(G)$ is the function with $\delta_h(h)=1$ and $\delta_h(g)=0, \forall g\ne h$. Note that $\rho(g_1g_2)=\rho(g_2)\rho(g_1)$. Here we choose right regular representation instead of left regular representation, since under this choice, $\rho(g)$ is $G$-invariant under the usual left $G$-action on $B(l^2(G))$.

\begin{lem}\label{lemstr1} Let $X$ be a left $G$-set, and $C[G\times X], C[G]$ be the sets as defined in Definition \ref{Roealg}. Then
	\begin{enumerate}
		\item
			$C[G\times X]$ is a $*$-subalgebra of $B(\mathcal H_{G\times X})$, the $C^*$-algebra of all bounded linear operators on $\mathcal H_{G\times X}$. 
		\item 
		    $C[G]\cong \rho(\mathbb C[G])\odot\mathcal K$, the algebraic tensor product of $\rho(\mathbb C[G])$ and $\mathcal K$, under the natural unitary isomorphism $U: \mathcal H_G\rightarrow l^2(G)\otimes\mathcal H,\ U(v)=\sum_{g\in G}{\delta_g}\otimes v(g)$.
	\end{enumerate}
\end{lem}

\begin{proof}
	(1) Let $S, T\in C[G\times X]$ and $a\in\mathbb C$. It is easy to see that $aT\in C[G\times X]$. Now $S+T\in C[G\times X]$ follows from the fact that supp$(S+T)\subseteq$supp$(S)\bigcup$ supp$(T)$, and $S\circ T\in C[G\times X]$ follows from the fact that elements of $C[G\times X]$ are row and column finite and Lemma \ref{T-properties}. Finally $S^*\in C[G\times X]$ also follows from Lemma \ref{T-properties}. 
	
	(2) Let $T\in C[G]$, then there is a finite subset $S\subseteq G$, so that $T_{g', g}\neq0$ only if $g'=gs$ for some $s\in S$, and $T_{gs, g}=T_{s, e}$ by $G$-invariance of $T$, where $e\in G$ is the identity elment. Now define a map $\sigma: C[G]\longrightarrow\rho(\mathbb C[G])\odot\mathcal K$ by $\sigma(T):=\sum_{s\in S}\rho(s)\otimes T_{s, e}$.
	We check that $\sigma$ is the $*$-algebra homomorphism that is induced by the natural unitary isomorphism $U$, i.e., $U\circ T=\sigma(T)\circ U$. It suffices to show that for any $g\in G, v\in\mathcal H$, we have $(U\circ T)(\iota_g(v))=(\sigma(T)\circ U)(\iota_g(v))$. We compute
	\begin{align*}
	   (U\circ T)(\iota_g(v))&=U\Big(\sum_{g'\in G}\iota_{g'}\big(\pi_{g'}(T(\iota_g(v)))\big)\Big)\\
	   &=\sum_{g'\in G}U\Big(\iota_{g'}\big(T_{g', g}(v)\big)\Big)\\
	   &=\sum_{g'\in G}\delta_{g'}\otimes T_{g', g}(v)\\
	   &=\sum_{s\in S}\delta_{gs}\otimes T_{s, e}(v)
	\end{align*}
	and
	\begin{align*}
	(\sigma(T)\circ U)(\iota_g(v))&=\sigma(T)(\delta_g\otimes v)\\
	&=\sum_{s\in S}(\rho(s)\otimes T_{s, e})(\delta_g\otimes v)\\
	&=\sum_{s\in S}\delta_{gs}\otimes T_{s, e}(v).
	\end{align*}
Hence $U\circ T=\sigma(T)\circ U$. $\sigma$ is onto since it has an inverse given by $\tau: \rho(\mathbb C[G])\odot\mathcal K\longrightarrow C[G]$,  $\tau(\rho(s)\otimes J):= T$, where $T$ is the operator on $\mathcal H_G$ with components $T_{gs, g}=J, \forall g\in G$ and $T_{g', g}=0$ if $g'\neq gs$ (of course we extend $\tau$ to all elements of $\rho(\mathbb C[G])\odot\mathcal K$ by linearity). This completes the proof.
\end{proof}

    Recall,  as in Definition \ref{Completed_Roealg}, $C^*(G\times X)$ is the norm closure of $C[G\times X]$ in $B(\mathcal H_{G\times X})$ and $C^*(G)$ is $C^*(G\times X)$ when $X$ is the one-point space. According to the lemma above, we have $C^*(G)\cong C^*_r(G)\otimes\mathcal K$, where $C^*_r(G)$ is the reduced group $C^*$-algebra of $G$. To analyze the structure of $C^*(G\times X)$ for general $X$, we introduce some notations first. For every finite subset $F$ of $X$, let $C[G\times X; F]$ denote the subalgebra of $C[G\times X]$ consisting of all operators $T$ so that supp($T$)$\subseteq G_F\times G_F$, where $G_F:=\{(g, gx)\in G\times X\ |\ g\in G, x\in F\}$, and $C^*(G\times X; F)$ be its closure in $B(\mathcal H_{G\times X})$.  Then $C^*(G\times X; E)$ is a $C^*$-subalgebra of  $C^*(G\times X; F)$ if $E\subseteq F$,  $C[G\times X]=\bigcup_{F\subseteq X, F \text{ is finite}}C[G\times X; F]$, and $C^*(G\times X)$ is the closure of $\bigcup_{F\subseteq X, F \text{ is finite}}C^*(G\times X; F)$ in $B(\mathcal H_{G\times X})$.  Therefore $C^*(G\times X)=\text{colim}_{F\subseteq X, F\text{ is finite}} C^*(G\times X; F)$.

    For an algebra  $A$, and a finite set $F$, let $M_F(A)$ denote the algebra consisting of all  matrices indexed by $F\times F$ with entires in the algebra $A$. If $E\subseteq F$, then $M_E(A)$ embeds in $M_F(A)$ naturally. 

\begin{lem}\label{lemstr2} With the above notations, we have
	\begin{enumerate}
		\item 
		     $C[G\times X; F]\cong M_F(C[G])\cong M_F(\rho(\mathbb C[G])\odot\mathcal K)$, therefore $C^*(G\times X; F)\cong M_F( C^*_r(G)\otimes\mathcal K)$;
		\item 
		     $C^*(G\times X)\cong \text{colim}_{F\subseteq X, F\text{\ is finite}} M_{F}\big(C^*_r(G)\otimes\mathcal K\big)$, where the direct colimit is taking over all finite subsets of $X$ under inclusions. Therefore, the $K$-theory of $C^*(G\times X)$ agrees with the $K$-theory of $C^*_r(G)$. 
    \end{enumerate}
\end{lem}
\begin{proof}
	(1) For any $T\in C[G\times X; F]$, we define $M^T\in M_F(C[G])$ to be the element whose $(x', x)$-entry, $x', x\in F$, is given by $(M^T_{x', x})_{g', g}=T_{(g', g'x'),(g, gx)}$. It is then not hard to check that the assignment $T\mapsto M^T$ gives an isomorphism between the $*$-algebras $C[G\times X; F]$ and $M_F(C[G])$. Combining Lemma \ref{lemstr1}, part (1) follows.
	
	(2) This follows from part (1) and standard facts from $K$-theory of $C^*$-algebras.
\end{proof}


Next, we consider the case when $X=G/H$, where $H<G$ is a subgroup, and study a subalgebra of $C^*(G\times G/H)$ that we now define. This will be relevant when we determine the value of $LK^G_*$ at $X=G/H$.

\begin{defn} 
	Let $H<G$ be a subgroup. Define $C[G\times G/H; 0]$ to be the subset of $C[G\times G/H]$ consisting of all operators $T$ with the following additional property: for any $g, g'\in G, x,x'\in G/H$, if $x\ne x'$, then $T_{(g', x'),(g, x)}=0$. $C[G\times G/H; 0]$ is a $*$-subalgebra of $C[G\times G/H]$. Its norm closure in
$B(\mathcal H_{G\times G/H})$ is denoted by $C^*(G\times G/H; 0)$.
\end{defn}

   The structure of the algebra $C^*(G\times G/H; 0)$ and its $K$-theory is described in the following lemma:

\begin{lem}\label{lemmorecontrol}  Let $H\backslash G$ be the right coset space of $H$ in $G$ with the trivial $H$-action. Then  
     $$C[G\times G/H; 0]\cong C[H\times H\backslash G],$$
     $$C^*(G\times G/H; 0)\cong C^*(H\times H\backslash G).$$
Therefore we have isomorphic $K$-groups $$K_i(C^*(G\times G/H; 0))\cong K_i(C^*_r(H)), i=0, 1.$$
\end{lem}
\begin{proof} 
	For each $x\in G/H$, let $\iota_x:\mathcal H_{G\times\{x\}}\rightarrow\mathcal H_{G\times G/H}$ be the obvious embedding, and $\pi_x:\mathcal H_{G\times G/H}\rightarrow H_{G\times\{x\}}$ be the obvious orthogonal projection. For any 
	$T\in C[G\times G/H;0]$, define $T^x:=\pi_x\circ T\circ\iota_x$.
	In terms of components, we have $T^x_{(g', x), (g, x)}=T_{(g', x),(g, x)}$. Due to the support condition on $T$, $T=\oplus_{x\in G/H} T^x$. By $G$-invariance of $T$, each $T^x$ is $G_x$-invariant, where $G_x$ is the stabilizer of $x$ for the $G$-action on $G/H$. For any $x,x'\in G/H$, and any $k\in G$ with $kx=x'$, we have $T^{x'}=\iota_k\circ T^x\circ\iota_{k^{-1}}$, where $\iota_k:\mathcal H_{G\times\{x\}}\rightarrow\mathcal H_{G\times\{x'\}}$ is the unitary isomorphism induced by $(g, x)\mapsto(kg, kx)=(kg, x')$. These arguments show that $T$ is completely determined by $T^x$ for any single $x\in G/H$ and $||T||=||T^x||$.
	 
	Let $x_0=H\in G/H$. Choose $g_\lambda\in G, \lambda\in\Lambda$, so that $H\backslash G=\{Hg_\lambda\ |\ \lambda\in\Lambda\}$. Let $j$ be the unitary isomorphism of $\mathcal H_{H\times H\backslash G}$ to $\mathcal H_{G\times\{x_0\}}$ induced by $(h, Hg_\lambda)\mapsto (hg_\lambda, x_0)$. For any $T\in C[G\times G/H; 0]$, let $M^T:=j^{-1}\circ T^{x_0}\circ j\in B(\mathcal H_{H\times H\backslash G})$. In terms of components, we have for any  $h, h'\in H, g_{\lambda}, g_{\lambda'}\in G, \lambda, \lambda'\in\Lambda$
	$$M^T_{(h', Hg_{\lambda'}),(h, Hg_\lambda)}=T_{(h'g_{\lambda'}, x_0), (hg_\lambda, x_0)}\in\mathcal K.$$

    We claim that $M^T\in C[H\times H\backslash G]$. Clearly $M^T$ is locally compact. $M^T$ is $H$-invariant because $H$ stabilizes $x_0$ and $T$ is $G$-invariant.  By definition of $T$, there exist  finite subsets $F\subseteq G/H, S\subseteq G$, so that $T_{(g', x'), (g, x)}\neq 0$ implies $g^{-1}x, g'^{-1}x'\in F$ and $g^{-1}g'\in S$. Let $F=\{x_1, x_2, \cdots, x_m\}$. Since the action of $G$ on $G/H$ is transitive and $H$ stabilizes $x_0$, we have for each $x_i\in F$, there exists a unique $\lambda_i\in\Lambda$, so that $g_{\lambda_i}^{-1}x_0=x_i, i=1, \cdots, m$. Let $R=\{g_{\lambda_1}, \cdots, g_{\lambda_m}\}$. 
    
    Now suppose $M^T_{(h', Hg_{\lambda'}),(h,Hg_\lambda)}=T_{(h'g_{\lambda'}, x_0), (hg_\lambda, x_0)}\neq 0$.
    Then $g^{-1}_\lambda h^{-1}h'g_{\lambda'}\in S$ and $g_\lambda^{-1}h^{-1}x_0=g_\lambda^{-1}x_0, g_{\lambda'}^{-1}h'^{-1}x_0=g_{\lambda'}^{-1}x_0\in F$. Therefore $g_\lambda, g_{\lambda'}\in R$, which is a finite set. Since the action of $H$ on $H\backslash G$ is trivial, this proves $M^T$ satisfies condition (3) in Definition \ref{Roealg}. $M^T$ also satisfies condition (4) in Defition \ref{Roealg} since $g^{-1}_\lambda h^{-1}h'g_{\lambda'}\in S$ implies $h^{-1}h'\in g_\lambda Sg_{\lambda'}^{-1}\subseteq RSR^{-1}$, and therefore $h^{-1}h'\in H\cap RSR^{-1}$, which is a finite subset of $H$. This completes the proof that $M^T\in C[H\times H\backslash G]$. Now it is easy to check that the assignment $T\mapsto M^T$ is a $*$-algebra homomorphism. It is norm preserving because $||M^T||=||T^{x_0}||=||T||$. 
    
    For any $M\in C[H\times H\backslash G]$, let $T^{x_0}=j\circ M\circ j^{-1}\in B(\mathcal H_{G\times\{x_0\}})$, and $T^x=\iota_k\circ T^{x_0}\circ\iota_{k^{-1}}\in B(\mathcal H_{G\times\{x\}})$, where $x, x_0\in G/H$ and $k\in G$ with $kx_0=x$. $T^x$ is independent of the choice of $k$ since $T^{x_0} $ is $H$-invariant. Now define $T^M:=\oplus_{x\in G/H}T^x\in B(\mathcal H_{G\times G/H})$. Similarly, one can show that $T^M\in C[G\times G/H; 0]$, and clearly $T^{M^T}=T$ and $M^{T^M}=M$.
    Therefore $C[G\times G/H; 0]\cong C[H\times H\backslash G]$ as normed $*$-algebras. Now the rest statements of the lemma follow directly from this and Lemma \ref{lemstr2}. 
\end{proof}


\subsection{The algebras $C_L[G\times X]$ and $C^*_L(G\times X)$} Next we show the set $C_L[G\times X]$ as defined in Definition \ref{localalg} is a $*$-subalgebra of $\mathcal{BU}\big([0, \infty), C^*(G\times X)\big)$, and 
therefore its norm closure $C^*_L(G\times X)$ is a $C^*$-algebra. 

We shall need the following basic facts about equivariantly continuously controlled sets as defined in Definition \ref{ecc}. Part (3) of the following lemma is not needed here, but will be needed when we prove the functoriality of $LK^G_*$.

\begin{lem}\label{ecc_properties} Let $X$ be a left $G$-space and $E, E'\in\mathcal E^G_X$. Then
    \begin{enumerate}
    \item $E\cup E'\in\mathcal E^G_X$;
    \item $E\circ E'\cup E'\circ E\in\mathcal E^G_X$, where
    $$E\circ E'=\{(x_1, x_2, t)\in X\times X\times[0, \infty)\ | \ \exists x\in X\ \textnormal{s.t.}\ (x_1, x, t)\in E, (x, x_2, t)\in E'\},$$ and $E'\circ E$ is defined similarly;
    \item if $\phi: X\rightarrow Y$ is a $G$-equivariant continuous map between two  $G$-CW-complexes and $K\subseteq X$ is compact, then the set 
    $$F=\{(\phi(x'), \phi(x), t)\ |(x', x, t)\in E, x', x\in G\cdot K\}\cup\{(y, y, t)\ |y\in Y,\ t\ge 0\}$$ 
    belongs to $\mathcal E^G_Y$.
    \end{enumerate} 
\end{lem}   
\begin{proof}
Parts (1) and (2) are easily verified. The proof of part (3) is very similar to that of \cite[Lemma 3.3]{BFJR}. It is  clear that $F$ is $G$-invariant, symmetric and contains the diagonal. We need to check condition (1) of Definition \ref{ecc} is satisfied by $F$. We prove by contradiction. Suppose not, then there exist $y\in Y$ and a $G_y$-invariant open neighborhood $U$ of $y$ so that for any $N>0$ and any $G_y$-invariant open neighborhood $V\subseteq U$ of $y$, there exist $(x', x, t)\in E$ with  $x', x\in G\cdot K$, $t>N$, and one of $\phi(x'), \phi(x)$ belongs to $V$, but the other does not belong to $U$. 

Since $Y$ is a $G$-CW complex, the \textit{slice theorem} (\cite[Proposition 3.4]{BFJR}) applies, so that we can find a descending sequence $\{V^k\}_{k\in\mathbb N}$ of open neighborhoods of $y$ in $Y$ with the following properties: 
\begin{itemize}
\item[(i)] each $V^k$  is $G_{y}$-invariant;

\item[(ii)] $gV^k\cap V^k=\emptyset$ if $g\notin G_{y}$;

\item[(iii)] $\bigcap_{k\ge 1}\overline{G\cdot V^k}=G\cdot y$.
\end{itemize}
We may assume $V^k\subseteq U, \forall k\in\mathbb N$. Then we can find, for each $k\in N$, $(x_k', x_k, t_k)\in E$ with
$x_k', x_k\in G\cdot K$, $t_k>k$, and one of  $\phi(x_k'), \phi(x_k)$ belongs to $V^k$, but the other does not belong to $U$. As $E$ is symmetric, we may assume $\phi(x_k)\in V^k, \phi(x_k')\notin U,  \forall k\in\mathbb N$.  As $x_k\in G\cdot K$ and $K$ is compact, by passing to a subsequence, we can find $g_k\in G$, so that $g_kx_k$ converges to a point, say $x$, in $X$. For a fixed $n\in\mathbb N$, we have, for all $k>n$,
$$\phi(g_kx_k)=g_k\phi(x_k)\in g_kV^k\subseteq g_kV^n\subseteq G\cdot V^n.$$
Hence, by letting $k\to\infty$, we have $\phi(x)\in\overline{G\cdot V^n}, \forall n\in\mathbb N$. So by property (iii) above, $\phi(x)\in G\cdot y$. Therefore there exists $g\in G$ so that $g\phi(x)=y$. Hence when $k$ is large enough, we have 
$$gg_k\phi(x_k)=g\phi(g_kx_k)\in V^1.$$
As $\phi(x_k)\in V^1$, by property (ii) above, we have, for $k$ large enough, $gg_k\in G_y$. Hence for large $k$, we have
\begin{align}\label{notinU}
gg_k\phi(x_k)\in V^k,\ gg_k\phi(x_k')\notin U.
\end{align}

Now consider $\phi^{-1}(U)$, which is a $G_{gx}$-invariant open neighborhood of $gx\in X$. As $E\in\mathcal E^G_X$, there exist $N>0$ and a $G_{gx}$-invariant open neighborhood $W\subseteq\phi^{-1}(U)$ of $gx$ satisfying condition (1) of Definition \ref{ecc}. As $gg_kx_k$ converges to $gx$ and $t_k$ tends to $\infty$, we have for $k$ large enough, $gg_kx_k\in W, t_k>N$. This, together with $(gg_kx_k', gg_kx_k, t_k)\in E$, implies $gg_kx_k'\in\phi^{-1}(U)$. Hence $gg_k\phi(x_k')\in U$ when $k$ is large enough, which contradicts to (\ref{notinU}). This completes the proof.
\end{proof}

\begin{lem}\label{localalglem} 
	$C_L[G\times X]$ is a $*$-subalgebra of $\mathcal{BU}\big([0, \infty), C^*(G\times X)\big)$. Hence $C^*_L(G\times X)$ is a $C^*$-algebra.
\end{lem}
\begin{proof}Let $f, g\in C_L[G\times X]$ and $c\in\mathbb C$ be a constant. It is easy to check that $cf, f+g, f^*\in C_L[G\times X]$. For $fg$, first it is still bounded and uniformly continuous. Now as in definition \ref{localalg}, let $f$ be $(S_1, K_1, \Lambda_1, E_1)$-controlled and $g$ be $(S_2, K_2, \Lambda_2, E_2)$-controlled. Suppose $(f(t)g(t))_{(g', x'), (g, x)}\ne 0$, then by Lemma \ref{T-properties}, there exists $(g'', x'')$ so that 
$f(t)_{(g', x'), (g'', x'')}\ne 0$ and $g(t)_{(g'', x''), (g, x)}\ne 0$. From here, we see that $fg$ is $(S_2S_1, K_1\cup K_2, \Lambda_1\cup\Lambda_2, (E_1\circ E_2)\cup(E_2\circ E_1))$-controlled, as $(E_1\circ E_2)\cup(E_2\circ E_1)\in\mathcal E^G_X$ by Lemma \ref{ecc_properties}, we get $fg\in C_L[G\times X]$. 
\end{proof}


\subsection{A $K$-theory lemma} Next we prove a very useful $K$-theory lemma that will be used several times later. It is a generalization of  \cite[Section 3, Lemma 2]{HRY}.

\begin{lem}\label{conjugate} Let $\Phi: \mathcal A\rightarrow\mathcal B$ 
	be a homomorphism of $C^*$-algebras, and $\mathcal C$ be a unital $C^*$-algebra which contains $\mathcal B$ as a $C^*$-subalgebra (not necessary as an ideal). Suppose $w\in\mathcal C$ is a partial isometry so that $$\Phi(a)w^*w=\Phi(a),\ w\Phi(a)w^*,\ w\Phi(a)\in\mathcal B$$ 
	for any $a\in\mathcal A$. Then Ad$(w)\circ\Phi(a)=w\Phi(a)w^*$ is a $C^*$-algebra homomorphism from $\mathcal A$ to $\mathcal B$, and on $K$-groups, we have $(Ad(w)\circ\Phi)_*=\Phi_*: K_*(\mathcal A)\rightarrow K_*(\mathcal B)$.
\end{lem}

\begin{proof}
	It is easy to see that $Ad(w)\circ\Phi$ is a $C^*$-algebra homomorphism form $\mathcal A$ to $\mathcal B$. To prove the second statement, we first consider the case when $w\in\mathcal C$ is a unitary. Let $j: \mathcal B\rightarrow M_2(\mathcal B)$ be the inclusion $$j(b)=
	\begin{bmatrix}
	b & 0\\
	0 & 0
	\end{bmatrix}$$
	We show $j\circ(Ad(w)\circ\Phi)$ is homotopic to $j\circ\Phi$, which will complete the proof since $j_*$ is an isomorphism on $K$-groups. 
	
	Let $$u(t)=
	\begin{bmatrix}
	\cos(\frac{\pi}{2}t)w & \sin(\frac{\pi}{2}t)\\
	-\sin(\frac{\pi}{2}t) & \cos(\frac{\pi}{2}t)w^*
	\end{bmatrix}$$
	Then $u(t)\in M_2(\mathcal C)$ is unitary for all $t\in\mathbb R$.
	Now define $$H_t(a)=u(t)(j(\Phi(a)))u(t)^*.$$ One checks, using the assumptions, that $H_t(a)\in M_2(\mathcal B)$ for all $a\in\mathcal A, t\in\mathbb R$.
	Hence $H_t$ is a homotopy from $H_0=j\circ(Ad(w)\circ\Phi)$ to
	$H_1$, which is given by 
	$$H_1(a)=\begin{bmatrix}
	0 & 0\\
	0 & \Phi(a)
	\end{bmatrix}$$
	A similar homotopy can be constructed between $H_1$ and $j\circ\Phi$, and this completes the proof when $w$ is a unitary.
	
	In general, we let 
	$$u=\begin{bmatrix}
	w & 1-ww^*\\
	w^*w-1 & w^*
	\end{bmatrix}$$
	Then $u\in M_2(\mathcal C)$ is a unitary. One checks, using the assumptions, that 
	$$u\cdot j(\Phi(a))\cdot u^*=j(w\Phi(a)w^*)\in M_2(\mathcal B), \ \ u\cdot j(\Phi(a))=j(w\Phi(a))\in M_2(\mathcal B)$$
	for any $a\in\mathcal A$. Hence by the special case that we have proved above, we get
	$$(Ad(u)\circ(j\circ\Phi))_*=(j\circ\Phi)_*$$
	On the other hand, $u\cdot j(\Phi(a))\cdot u^*=j(w\Phi(a)w^*)$ implies
	$$Ad(u)\circ(j\circ\Phi)=j\circ(Ad(w)\circ\Phi)$$
	Therefore $$(j\circ\Phi)_*=(j\circ(Ad(w)\circ\Phi))_*$$
	which implies $\Phi_*=(Ad(w)\circ\Phi)_*$ since $j_*$ is an isomorphism. This completes the proof of the lemma.
\end{proof}

 \subsection{The algebras $C_L^*(G\times X; A), C_L^*(G\times X; \langle A\rangle)$}

    We now turn to prove the following lemma. It is basic to many of our constructions and arguments. To state it, we recall that for a $G$-invariant subspace $A\subseteq X$ of a left $G$-space $X$, the sets $C_L[G\times X; A], C_L[G\times X; \langle A\rangle], C_L^*(G\times X; A), C_L^*(G\times X; \langle A\rangle)$ have been defined in Definition \ref{local_alg_subspace}.

\begin{lem}\label{ideal} Let $X$ be a left $G$-space and $A\subseteq X$ be a $G$-invariant subspace. 
	Then
	\begin{enumerate}
		\item $C_L[G\times X; A]$ and $C_L[G\times X; \langle A\rangle]$ are $*$-subalgebras of $C_L[G\times X]$, and $C_L[G\times X; A]\subseteq C_L[G\times X; \langle A\rangle]$;
		\item $C_L^*(G\times X; \langle A\rangle)$ is a two-sided ideal of $C^*_L(G\times X)$;
		\item If $A$ is a closed subset of $X$, then $C^*_L(G\times X; A)$ can be naturally identified with $C^*_L(G\times A)$.
		\item If $A$ is a closed subset of $X$, then the inclusion induces an isomorphism on the level of $K$-groups  
		$$K_i(C^*_L(G\times A))\cong K_i(C^*_L(G\times X; \langle A\rangle)), i=0, 1.$$
	\end{enumerate}
\end{lem}
\begin{proof}
	(1) We only check that $C_L[G\times X; \langle A\rangle]$ is a $*$-subalgebra. The only nontrivial part is to show $C_L[G\times X; \langle A\rangle]$ is closed under multiplcation. Indeed something stronger holds.  Let $f\in C_L[G\times X; \langle A\rangle], g\in C_L[G\times X]$.  Suppose $$\big(f(t)g(t)\big)_{(g', x'), (g, x)}=\sum_{(g'', x'')\in G\times X}f(t)_{(g', x'), (g'', x'')}g(t)_{(g'', x''), (g, x)}\neq 0,$$
	then there exists $(g'', x'')\in G\times X$ so that $f(t)_{(g', x'), (g'', x'')}\neq 0$ and $ g(t)_{(g'', x''), (g, x)}\neq 0$. By definition,  there exist $E_1, E_2\in\mathscr E^G_X$, compact $K\subseteq X$, finite $S\subseteq G$ and $a', a''\in A$,  so that $(a', x', t), (a'', x'', t)\in E_1$, $(x'', x, t)\in E_2$, $g^{-1}g''\in S$ and $(g')^{-1}a', (g'')^{-1}a'', (g')^{-1}x', (g)^{-1}x\in K$. Therefore $(a'', x, t)\in E_1\circ E_2\subseteq E_1\circ E_2\cup E_2\circ E_1\in \mathscr E^G_X$ and $g^{-1}a''\in SK$. Let $E=E_1\cup E_1\circ E_2\cup E_2\circ E_1\in\mathscr E^G_X$ and $K'=K\cup SK$. Then
	$K'$ is compact,  $(a', x', t), (a'', x, t)\in E$ and $(g')^{-1}a', g^{-1}a'', (g')^{-1}x', g^{-1}x\in K'$. Therefore $fg\in C_L[G\times X; \langle A\rangle]$.  
	Similar argument shows $gf\in C_L[G\times X; \langle A\rangle]$. Now $C_L[G\times X; A]\subseteq C_L[G\times X; \langle A\rangle]$ follows from  $\Delta\in\mathscr E^G_X$.
	
	(2) This follows from the proof of part (1).
	
	(3) We first note that $\mathscr E^G_A=\iota^{-1}\mathscr E^G_X$, where $\iota: A\times A\times [0, \infty)\rightarrow X\times X\times [0, \infty)$ is the inclusion. This follows easily from the definition and the fact that every $G_a$-invariant open neighborhood of $a\in A$ in $A$ is of the form $U\cap A$, where $U$ is a $G_a$-invariant open neighborhood of $a$ in $X$. It gives a natural inclusion $C_L[G\times A]\rightarrow C_L[G\times X; A]$. Now this inclusion is also surjective since $\mathscr E^G_A=\iota^{-1}\mathscr E^G_X$ and $A\cap K$ is compact for compact $K\subseteq X$ if $A$ is closed in $X$.
	
	(4) The proof here is a bit lengthy. For any $E\in\mathscr E_X^G$, and compact $K\subseteq X$, define $C_L[G\times X; \langle A\rangle, K, E]$ to be the set of all $f\in C_L[G\times X]$ with the property that if $f(t)_{(g',x'),(g, x)}\ne 0$, then $(g')^{-1}x', g^{-1}x\in K$ and there exist $a, a'\in A$ so that $g^{-1}a, (g')^{-1}a'\in K$ and
	$(a, x, t), (a', x', t)\in E$. We also define $C_L[G\times X; A, K]$
	to be the subset of $f\in C_L[G\times X; A]$ with the property that $f(t)_{(g', x'), (g, x)}\ne 0$ implies $g^{-1}x, (g')^{-1}x'\in K$.
	
	It is easy to check that $C_L[G\times X; A,K]$ and $C_L[G\times X; \langle A\rangle, K, E]$ are $*$-subalgebras of $C_L[G\times X; \langle A\rangle]$, and
	$C_L[G\times X; A, K]\subseteq C_L[G\times X; \langle A\rangle, K, E]$ since $\Delta\subseteq E$. Let $C^*_L(G\times X; A, K)$ and $C^*_L(G\times X; \langle A\rangle, K, E)$ be their closures in $C^*_L(G\times X)$. By definition of $C_L[G\times X; \langle A\rangle]$, we have 
	$$C_L^*(G\times X; \langle A\rangle)=\text{colim}_{K\subseteq X\ \text{compact},\ E\in\mathscr E^G_X}C^*_L(G\times X; \langle A\rangle, K, E)$$
	where the direct colimit is taking over all compact $K\subseteq X$, $E\in\mathscr E^G_X$ under inclusions. We also have
	$$C^*_L(G\times X; A)=\text{colim}_{K\subseteq X\ \text{compact}}C^*_L(G\times X; A, K)$$
	Therefore, it suffices to show that the inclusion
	$$C^*_L(G\times X; A, K)\longrightarrow C^*_L(G\times X; \langle A\rangle, K, E)$$
	induces isomorphisms on $K$-groups for any $E\in\mathscr E^G_X$.
	
	Let $B\subseteq X$ be any $G$-invariant subspace with the property that
	$A\subseteq B$ and $B-A$ is countable. 
	Let 
	$$C_L[G\times X; \langle A\rangle, K, E, B]=C_L[G\times X; \langle A\rangle, K, E]\cap C_L[G\times X; B]$$
	By the definition of elements of $C_L[G\times X]$, we have
	$$C_L[G\times X; \langle A\rangle, K, E]=\bigcup_{B\subseteq X}C_L[G\times X; \langle A\rangle, K, E, B]$$
	where the union is taking over all $G$-invariant subspaces $B\subseteq X$ with the property that $A\subseteq B$ and $B-A$ is countable. Therefore
	$$C_L^*(G\times X; \langle A\rangle, K, E)=\text{colim}_{B\subseteq X}C^*_L(G\times X; \langle A\rangle, K, E, B)$$
	where the direct colimit is taking over all $B$ with the property mentioned above. Therefore, it suffices to show the inclusion 
	$$\iota: C_L^*(G\times X; A, K)\longrightarrow C^*_L(G\times X; \langle A\rangle, K, E, B)$$
	induces isomorphisms on $K$-groups for every compact $K\subseteq X, E\in\mathscr E^G_X$ and $G$-invariant $B\subseteq X$ with $A\subseteq B$ and $B-A$ countable.
	
	Now fix $K, E, B$. We may assume $K\cap A\ne\emptyset$, otherwise both of  $C_L^*(G\times X; A, K)$ and $C^*_L(G\times X; \langle A\rangle, K, E, B)$ are trivial. For each $t\ge 0$, let 
	$$E_t=\{b\in B\cap K\ |\ \exists\ a\in A\cap K,\ s\ge t-1,\ \text{so that}\ (a, b, s)\in E\}$$
	Clearly $E_t\subseteq E_{t'}$ if $t\ge t'$, and $A\cap K\subseteq E_t$ for each $t\ge 0$ since $\Delta\subseteq E$. Choose a map
	$$\theta_t: B\longrightarrow A$$ 
	so that  $\theta_t(a)=a$ if $a\in A$, and if $b\in E_t$, then $\theta_t(b)\in A\cap K, (\theta_t(b), b, s)\in E$ for some $s\ge t-1$.
    Note that for each $a\in A$, $\theta_t^{-1}(a)$ is countable. Therefore, we can choose, for each $a\in A$, an isometry of Hilbert spaces:
	$$W_{t, a}: \mathcal H_{\theta_t^{-1}(a)}\longrightarrow \mathcal H_{\{a\}}$$
	which gives an isometry of Hilbert spaces:
	$$W_t: \mathcal H_{B}\longrightarrow\mathcal H_{A}$$
	$W_t$ gives rise to an isometry of Hilbert spaces
	$$V_t: \mathcal H_{G\times B}\longrightarrow\mathcal H_{G\times A}$$
	which is uniquely determined by the properties that
	$$V_t(\mathcal H_{\{g\}\times B})\subseteq\mathcal H_{\{g\}\times A},\ \ \ V_t|_{\mathcal H_{\{e\}\times B}}=W_t$$
	and $V_t$ is $G$-invariant. $V_t$ can be extended, in the obvious way, to a $G$-invariant partial isometry from $\mathcal H_{G\times X}$ to $\mathcal H_{G\times X}$. We still denote this partial isometry by $V_t$.
	

	By definitions of $V_t$ and $\theta_t$, we have 
	$$(V_t)_{(g', x'), (g, x)}\ne 0\ \Rightarrow\ g=g',\ x\in B,\ x'\in A,\  \theta_t(g^{-1}x)=g^{-1}x'$$
	If in addition $g^{-1}x\in E_t$, then $g^{-1}x'\in A\cap K$, and $ (x',x, s)\in E$ for some $s\ge t-1$.

	Now 
	define a family of $G$-invariant partial isometries $U_t, t\ge 0$ from $\mathcal H_{G\times X}\oplus\mathcal H_{G\times X}$ to $\mathcal H_{G\times X}\oplus\mathcal H_{G\times X}$ by
	$$U_t=R(t-k+1)(V_{k-1}\oplus V_{k})R^*(t-k+1),\ k-1\le t<k,\ k=1, 2, \cdots$$
	where $$R(t)=
	\begin{bmatrix}
	\ \ \cos(\frac{\pi t}{2}) & \sin(\frac{\pi t}{2})\\
	-\sin(\frac{\pi t}{2}) & \cos(\frac{\pi t}{2})
	\end{bmatrix}$$
	Notice that $U_t$ is uniformly continuous on $[k-1, k)$ and 
	$$\lim_{t\to k^{+}}U_t=U_k=V_k\oplus V_{k+1},\ \ \ \  \lim_{t\to k^{-}}U_t=V_k\oplus V_{k-1}$$
	For any $f\in C_L[G\times X; \langle A\rangle, K, E, B]$, define
	$$\big(\text{Ad}(U)\circ j\big)(f)(t)=U_t(f(t)\oplus 0)U^*_t,\ t\ge 0$$ 
	where $j(f)(t)=f(t)\oplus 0$.
	We shall prove the following claim:
	\vskip 10pt
	\noindent
	\underline{Claim}:
	\begin{enumerate}
		\item If $f\in C_L[G\times X; \langle A\rangle, K, E, B]$, then
		$(f(t)\oplus 0)U_t^*U_t=f(t)\oplus 0$,
		$U_t(f(t)\oplus 0)\in M_2\big(C_L[G\times X; \langle A\rangle, K, E, B]\big)$ and $U_t(f(t)\oplus 0)U_t^*\in M_2\big(C_L[G\times X; A, K]\big)$;
		\item If $f\in C_L[G\times X; A, K]$, then $U_t(f(t)\oplus 0)\in M_2\big(C_L[G\times X; A, K]\big)$.
	\end{enumerate}
    
    \vskip 10pt
    Assuming the claim, part (4) of the Lemma follows. To see this,  let $\mathcal B$ be the $C^*$-algebra of all bounded functions from $[0, \infty)$ to $B(\mathcal H_{G\times X})$. Then $U\in M_2(\mathcal B)$, where $U(t)=U_t, t\ge 0$, and $M_2\big(C^*_L(G\times X; \langle A\rangle, K, E,B)\big)\subseteq M_2(\mathcal B)$.
    By the claim above, Lemma \ref{conjugate} applies to both of
    $$\text{Ad}(U)\circ j: C^*_L(G\times X; A, K)\longrightarrow M_2\big(C^*_L(G\times X; A, K)\big)$$ 
    and 
    $$\text{Ad}(U)\circ j: C^*_L(G\times X; \langle A\rangle, K, E, B)\longrightarrow M_2\big(C^*_L(G\times X; \langle A\rangle, K, E, B)\big)$$
    where, in the first map, we view $j$ as an inclusion from $C^*_L(G\times X; A, K)$ to $M_2\big(C^*_L(G\times X; A, K)\big)$, and in the second map,  we view $j$ as an inclusion from $C^*_L(G\times X; \langle A\rangle, K, E, B)$ to $M_2\big(C^*_L(G\times X; \langle A\rangle, K, E, B)\big)$.
    Therfore, on the level of $K$-groups, we have in both cases
    $$\big(\text{Ad}(U)\circ j\big)_*=j_*$$
    Now using the sequence of maps
    $$C^*_L(G\times X; A, K)\rightarrow C^*_L(G\times X;\langle A\rangle, K, E, B)\rightarrow $$
    $$M_2\big(C^*_L(G\times X; A, K)\big)\rightarrow M_2\big(C^*_L(G\times X; \langle A\rangle, K, E, B)\big)$$
    and the fact that $j_*$ is an isomorphism, we see that the inclusion 
    $$\iota: C^*_L(G\times X; A, K)\longrightarrow C^*_L(G\times X; \langle A\rangle, K, E, B)$$ induces an isomorphism on the level of $K$-groups. 
    
    We now prove the claim.
    First, $(f(t)\oplus 0)U_t^*U_t=f(t)\oplus 0$ follows from a direct computation and the fact that $V_k, f(t)$ are supported on $G\times B$ and $V_k$ is an isometry when restricted to $\mathcal H_{G\times B}$. Using $U_t$ is uniformly continuous on eact interval $[k-1, k)$ and 
    $$\lim_{t\to k^{+}}U_t=U_k=V_k\oplus V_{k+1},\ \ \ \  \lim_{t\to k^{-}}U_t=V_k\oplus V_{k-1}$$
    we see that every matrix entry of $U_t(f(t)\oplus 0)$ and $U_t(f(t)\oplus 0)U_t^*$ is uniformly continuous in $t\in [0, \infty)$(even though $U_t$ is discontinuous at $t=1, 2, \cdots$).
    
    Fix $t\ge 0$, and let $k\in\mathbb N$ so that $k-1\le t<k$. One computes that every matrix entry of $U_t$ is of the form $A(t)=a(t)V_{k-1}+b(t)V_k$ for some $a(t), b(t)\in\mathbb R$.  Suppose $A(t)_{(g', x'),(g, x)}\ne 0$. Then by the support conditions 
    of $V_k$ and $V_{k-1}$, we have
    $$g=g',\ \ x, x'\in B$$
    If in addition $g^{-1}x\in E_k\subseteq E_{k-1}$, then
    $$g^{-1}x'\in\{\theta_{k-1}(g^{-1}x), \theta_{k}(g^{-1}x)\}\subseteq A\cap K, \ (x', x, t)\in\overline E$$ 
    where $$\overline E=\{(x', x, t)\in X\times X\times[0, \infty)\ |\ \exists s\ge t-2,\ \text{so that}\ (x', x, s)\in E\ \}$$
    It is not hard to show that $\overline E\in\mathscr E^G_X$. 
   
    Assume $f\in C_L[G\times X; \langle A\rangle, K, E, B]$. One computes that every matrix entry of $U_t(f(t)\oplus 0)$ is of the form $A(t)f(t)$. Since $f(t)$ is row and column finite and locally compact, Lemma \ref{T-properties} applies and we get $A(t)f(t)$ is locally compact and $$\big(A(t)f(t)\big)_{(g', x'),(g, x)}\ne 0$$
    implies there exists $(g'', x'')$ so that
    $$A(t)_{(g',x'), (g'', x'')}\ne 0,\ \ f(t)_{(g'', x''), (g, x)}\ne 0$$
    Hence 
    $$g'=g'',\ x'\text{\ and\ } x''\in B, \ g^{-1}x \text{\ and\ } (g'')^{-1}x''\in B\cap K$$
    and there exist a finite subset $F\subseteq X$ (depending on $t$), a finite subset $S\subseteq G$ (independent of $t$),  a countable $G$-invariant subset $\Lambda\subseteq X$ (independent of $t$), $a'', a\in A$ and $E'\in\mathscr E^G_X$ so that  
    $$g^{-1}x\text{\ and\ }(g'')^{-1}x''\in F,\ \ g^{-1}g''\in S,\ \ x\ \text{and\ } x''\in\Lambda$$
    $$\ g^{-1}a \text{\ and\ } (g'')^{-1}a''\in A\cap K,\ (a'', x'', t), (a, x, t)\in E,\ (x'', x, t)\in E' $$
    Therefore 
    $$g^{-1}g'\in S, \ g^{-1}x\text{\ and\ }(g'')^{-1}x''\in E_{k}$$
    Hence 
    $$(g')^{-1}x'\in\{\theta_{k-1}((g')^{-1}x''), \theta_k((g')^{-1}x'')\}\subseteq A\cap K, \ \ (x', x'', t)\in\overline{E}$$
    Thus
    $$(g')^{-1}x'\in\theta_{k-1}(F)\cup\theta_{k}(F),\ \ (g')^{-1}x'\in\bigcup_{i=0}^{\infty}\theta_i(\Lambda),\ \ (x', x, t)\in\overline{E}\circ E'\cup E'\circ\overline{E}\in\mathscr E^G_X$$
    Let $$\overline{F}=F\cup\theta_{k-1}(F)\cup\theta_{k}(F),\ \overline{\Lambda}=\Big(G\cdot\big(\bigcup_{i=0}^{\infty}\theta_i(\Lambda)\big)\Big)\cup\Lambda$$
    Then $\overline F\subseteq X$ is finite and $\overline{\Lambda}\subseteq X$ is countable and $G$-invariant, and the arguments above show that 
    $$\big(A(t)f(t)\big)_{(g', x'),(g, x)}\ne 0$$ implies
    $$g^{-1}x\text{\ and\ }(g')^{-1}x'\in\overline F,\ \ g^{-1}g'\in S,\ \ x'\text{\ and\ } x\in\overline{\Lambda}$$
    $$g^{-1}x\in B\cap K,\ \ (g')^{-1}x'\in A\cap K, \ \ (x', x, t)\in\overline{E}\circ E'\cup E'\circ\overline E$$
    where $S$ and $\overline{\Lambda}$ are independent of $t$, and the existence of $a, a'\in A$ (choose $a'=x'$) with the property that 
    $$g^{-1}a\text{\ and \ } (g')^{-1}a'\in A\cap K,\ \ (a, x, t), (a', x', t)\in E$$
    Therefore for $f\in C_L[G\times X; \langle A\rangle, K, E, B]$, we have 
    $$A(t)f(t)\in C_L[G\times X; \langle A\rangle, K, E, B]$$
    Hence $$U_t(f(t)\oplus 0)\in M_2\big(C_L[G\times X; \langle A\rangle, K, E, B]\big)$$
    
    The above arguments also show that if $f\in C_L[G\times X; A, K]$, then 
    $$U_t(f(t)\oplus 0)\in M_2\big(C_L[G\times X; A, K]\big)$$
    which is part (2) of the Claim.
    
    It remains to prove $U_t(f(t)\oplus 0)U_t^*\in M_2\big(C_L[G\times X; A, K]\big)$ for $f\in C_L[G\times X; \langle A\rangle, K, E, B]$. 
    One computes that every matrix entry of $U_t(f(t)\oplus 0)U_t^*$ is of the form $A(t)f(t)B(t)^*$, where $A(t), B(t)$ are matrix entries of $U_t$. We have already seen that $A(t)f(t)\in C_L[G\times X; \langle A\rangle, K, E, B]$. Therefore $$A(t)f(t)B(t)^*=\Big(B(t)\big(A(t)f(t)\big)^*\Big)^*\in C_L[G\times X; \langle A\rangle, K, E, B]$$
    
    Now suppose $$\big(A(t)f(t)B(t)^*\big)_{(g',x'),(g,x)}\ne 0$$
    Then there exist $(g'', x''), (g''',x''')$ so that 
    $$A(t)_{(g',x'),(g'',x'')}\ne 0,\ \ f(t)_{(g'',x''),(g''',x''')}\ne 0,\ \ B(t)_{(g,x),(g''',x''')}\ne 0$$
    Hence
    $$g'=g'',\ g=g''',\ (g'')^{-1}x''\ \text{and}\ (g''')^{-1}x'''\in E_k$$
    Therefore
    $$(g')^{-1}x'\in A\cap K,\ g^{-1}x\in A\cap K$$
    Hence  
    $$A(t)f(t)B(t)^*\in C_L[G\times X; A, K]$$
    This shows $U_t(f(t)\oplus 0)U_t^*\in M_2\big(C_L[G\times X; A, K]\big)$ and completes the proof of the claim.
\end{proof}

\subsection{Coarsely excisive triad}

    Next we define the notion of \textit{coarsely excisive triad}. Similar notions have been introduced and studied in different contexts such as \cite{HRY}\cite{HPR}\cite{BFJR}. We define it in our context. 
    

\begin{defn} Let $X_1$ and $X_2$ be two $G$-invariant subspaces of the 
	$G$-space $X$. The triad $(X, X_1, X_2)$ is called \textit{coarsely excisive} if $X=X_1\cup X_2, X_1\cap X_2\ne\emptyset$, and for any $E\in\mathscr E^G_X$, compact $K\subseteq X$, there exist $E'\in\mathscr E^G_X$ and compact $K'\subseteq X$ with the following property: 
	for any $g\in G, x\in X, t\ge 0$ with $g^{-1}x\in K$, if there exist $x_1\in X_1, x_2\in X_2$ with $(x_1, x, t), (x_2, x, t)\in E, g^{-1}x_1, g^{-1}x_2\in K$, then we can find $x_0\in X_1\cap X_2$ so that $(x_0, x, t)\in E'$ and $g^{-1}x_0\in K'$.
\end{defn}

    Intuitively, if $(X, X_1, X_2)$ is coarsely excisive, then a point $x\in X$ that is close to both of $X_1$ and $X_2$ at time $t$ must also close to $X_1\cap X_2$ at that time. We have the following important lemma:

\begin{lem}\label{excisivetriad}
	Let $X_1$ and $X_2$ be two $G$-invariant subspaces of the $G$-space $X$.  
	\begin{enumerate}
	\item If $X=X_1\cup X_2$, then 
	$$C^*_L(G\times X)=C_L^*(G\times X; \langle X_1\rangle)+C^*_L(G\times X; \langle X_2\rangle).$$ 
	\item If $(X, X_1, X_2)$ is coarsely excisive, then
	$$C^*_L(G\times X; \langle X_1\cap X_2\rangle)=C^*_L(G\times X; \langle X_1\rangle)\cap C^*_L(G\times X; \langle X_2\rangle).$$	
	Therefore we have an isomorphism of $C^*$-algebras induced by inclusions
	$$C^*_L(G\times X; \langle X_1\rangle)/C^*_L(G\times X; \langle X_1\cap X_2\rangle)\tilde{\longrightarrow}C^*_L(G\times X)/C^*_L(G\times X; \langle X_2\rangle).$$
	\item If $X$ is the disjoint union of $X_1$ and $X_2$ with $X_1, X_2$ open, then the K-groups of $C^*_L(G\times X; \langle X_1\rangle)\cap C^*_L(G\times X; \langle X_2\rangle)$ vanish. 
    \end{enumerate}
\end{lem}
    

\begin{proof}
	(1) For any $f\in C_L[G\times X]$, define $f_1$ by $f_1(t)_{(g',x'), (g, x)}=f(t)_{(g',x'), (g, x)}$ if $x, x'\in X_1$, and 
	$f_1(t)_{(g',x'), (g, x)}=0$ otherwise. Then $f_1\in C_L[G\times X; X_1]\subseteq C_L[G\times X; \langle X_1\rangle]$. Let $f_2=f-f_1$. We
	show $f_2\in C_L[G\times X; \langle X_2\rangle]$. Suppose $f_2(t)_{(g',x'), (g,x)}\ne 0$. Then $f_2(t)_{(g',x'), (g,x)}=f(t)_{(g',x'), (g,x)}\ne 0$ and at least one of $x, x'\in X_2$.
	By definition, there exist $E\in\mathscr E^G_X$, finite $S\subseteq G$ and compact $K\subseteq X$ so that $(x',x, t)\in E, g^{-1}g'\in S$ and $g^{-1}x, (g')^{-1}x'\in K$. By enlarging $S$, we may assume $S=S^{-1}$. If $x\in X_2$, we let $a'=a=x\in X_2$. Then $(a', x', t), (a, x, t)\in E$, $(g')^{-1}a'=(g')^{-1}x=(g')^{-1}gg^{-1}x\in S^{-1}K=SK, g^{-1}a=g^{-1}x\in K$. Let $K'=SK\cup K$. Then $g^{-1}x, g^{-1}a, (g')^{-1}x', (g')^{-1}a'\in K'$. Similar argument also applies to the case when $x'\in X_2$. Altogether, we have shown that $f_2\in C_L[G\times X; \langle X_2\rangle]$. Therefore 
	$$C_L[G\times X]=C_L[G\times X; \langle X_1\rangle]+C_L[G\times X; \langle X_2\rangle]$$
	Hence, by \cite[Section 3, Lemma 1]{HRY}, we have  
	$$C^*_L(G\times X)=C_L^*(G\times X; \langle X_1\rangle)+C^*_L(G\times X; \langle X_2\rangle)$$
	
	\vskip 10pt
	\noindent
	(2) Clearly we have $C^*_L(G\times X; \langle X_1\cap X_2\rangle)\subseteq C^*_L(G\times X; \langle X_1\rangle)\cap C^*_L(G\times X; \langle X_2\rangle)$. For the converse, it suffices to show $C_L[G\times X; \langle X_1\rangle]\cap C_L[G\times X; \langle X_2\rangle]\subseteq C_L[G\times X; \langle X_1\cap X_2\rangle]$ (it is not hard to see the closure of $C_L[G\times X; \langle X_1\rangle]\cap C_L[G\times X; \langle X_2\rangle]$ is equal to $C^*_L(G\times X; \langle X_1\rangle)\cap C^*_L(G\times X; \langle X_2\rangle)$). Let $f\in C_L[G\times X; \langle X_1\rangle]\cap C_L[G\times X; \langle X_2\rangle]$. By definition, for $i=1, 2$, there exist $E_i\in\mathscr E^G_X$, compact $K_i\subseteq X$ so that $f(t)_{(g', x'),(g, x)}\ne 0$ implies $g^{-1}x, (g')^{-1}x'\in K_i$, and the existence of $x_i, x_i'\in X_i$ with  $g^{-1}x_i, (g')^{-1}x_i'\in K_i$ and $(x_i, x, t), (x_i', x', t)\in E_i$. Let $E=E_1\cup E_2$, $K=K_1\cup K_2$. Then, since $(X, X_1, X_2)$ is coarsely excisive,  there exist $E'\in\mathscr E^G_X$ and compact $K'\subseteq X$ so that $f(t)_{(g', x'), (g, x)}\ne 0$ implies the existence of $x_0, x_0'\in X_1\cap X_2$ with the property $(x_0, x, t), (x_0', x', t)\in E'$ and $g^{-1}x_0, (g')^{-1}x_0'\in K'$. Hence $f\in C_L[G\times X; \langle X_1\cap X_2\rangle]$. This completes the proof of part (2).
	
	
	\vskip 10pt
	\noindent
	(3) We first show for any $f\in C_L[G\times X; \langle X_1\rangle]\cap C_L[G\times X; \langle X_2\rangle]$, there exists $t_f\ge 0$ so that $f(t)=0$ if $t\ge t_f$.
	For any $E\in\mathscr E^G_X$ and compact $K\subseteq X$, we let
	$$E_K=\{(x', x, t)\in E\ |\ x, x'\in G\cdot K\}.$$
	
	By assumption, $X_1$ and $X_2$ are $G$-invariant open subsets of $X$. Therefore, by definition of $E$, for any $a\in K\cap X_i, i=1, 2$, there exist a $G_{a}$-invariant open neighborhood $V_a$ of $a$ in $X_i$ and $t_{a}>0$ so that if $(x', x, t)\in E_K$ and $t>t_a$, then one of $x, x'$ lies in $V_a$ implies the other lies in $X_i$. Since $K$ is compact, we can find finitely many $a_1, a_2, \cdots, a_n\in K$ so that $K\subseteq\bigcup_{j=1}^{n}V_{a_j}$. Let $t_{E, K}=t_{a_1}+\cdots+t_{a_n}$. Assume $(x', x, t)\in E_K, t>t_{E, K}$ and $x\in X_i$. There exists $g\in G$ so that $gx\in K$. Hence $gx\in V_{a_j}\subseteq X_i$ for some $1\le j\le n$. Therefore $gx'\in X_i$. So $x'\in X_i$. This shows that there exists $t_{E, K}>0$, so that for any $(x', x, t)\in E_K$, if $t>t_{E, K}$, then either $x,x'\in X_1$ or $x, x'\in X_2$.
	
	Now let $f\in C_L[G\times X; \langle X_1\rangle]\cap C_L[G\times X; \langle X_2\rangle]$.  By definition, there exist $E\in\mathscr E^G_X$ and compact $K\subseteq X$ so that $f(t)_{(g', x'),(g, x)}\ne 0$ implies $g^{-1}x, (g')^{-1}x'\in K$ and the existence of $a, a'\in X_1$ with $(a, x, t), (a', x', t)\in E$ and $g^{-1}a, (g')^{-1}a'\in K$. This implies $(a, x, t), (a', x', t')\in E_K$. Therefore if $f(t)_{(g', x'), (g, x)}\ne 0$ and $t>t_{E, K}$, then $x, x'\in X_1$ since $a, a'\in X_1$. The same argument also applies to $X_2$. Since $X_1\cap X_2=\emptyset$, we see that there exists $t_f\ge 0$ so that $f(t)=0$ if $t\ge t_f$.
	
	For any $\alpha\ge 0$, let $C_L[G\times X; \langle X_1\rangle\cap\langle X_2\rangle; \alpha]$ be the $*$-subalgebra of $C_L[G\times X; \langle X_1\rangle]\cap C_L[G\times X; \langle X_2\rangle]$ that consists of all $f$ with the property that $f(t)=0$ if $t\ge\alpha$.  By the previous argument, we  have
	$$C_L[G\times X; \langle X_1\rangle]\cap C_L[G\times X; \langle X_2\rangle]=\bigcup_{\alpha\ge 0}C_L[G\times X; \langle X_1\rangle\cap\langle X_2\rangle; \alpha]$$
	Therefore
	$$C_L^*(G\times X; \langle X_1\rangle)\cap C_L^*(G\times X; \langle X_2\rangle)=\text{colim}_{\alpha\ge 0}C^*_L(G\times X; \langle X_1\rangle\cap\langle X_2\rangle; \alpha)$$
	where $C^*_L(G\times X; \langle X_1\rangle\cap\langle X_2\rangle; \alpha)$ is the closure of $C_L[G\times X; \langle X_1\rangle\cap\langle X_2\rangle; \alpha]$ in $C^*_L(G\times X)$, and the colimit is taking under the inclusion $C^*_L(G\times X; \langle X_1\rangle\cap\langle X_2\rangle; \alpha)\subseteq C^*_L(G\times X; \langle X_1\rangle\cap\langle X_2\rangle; \beta)$ for any $0\le\alpha\le\beta$. We thus have
	$$K_*\big(C_L^*(G\times X; \langle X_1\rangle)\cap C_L^*(G\times X; \langle X_2\rangle)\big)=\text{colim}_{\alpha\ge 0}K_*\big(C^*_L(G\times X; \langle X_1\rangle\cap\langle X_2\rangle; \alpha)\big).$$
	Fix $\alpha\ge 0$.
	For each $s\in[0, 1]$, we define 
	$$H_s: C^*_L(G\times X; \langle X_1\rangle\cap\langle X_2\rangle;\alpha)\longrightarrow C^*_L(G\times X; \langle X_1\rangle\cap\langle X_2\rangle;\alpha)$$
	$$H_s(f)(t)=f(t+s\alpha).$$
	$H_s(f)$ is continuous in $s$ since $f$ is uniformly continuous in $t$. Therefore $H_s$ is a homotopy between $H_0=id$ and $H_1=0$. Hence $K_*\big(C^*_L(G\times X; \langle X_1\rangle\cap\langle X_2\rangle; \alpha)\big)=0$. It follows that $K_*\big(C_L^*(G\times X; \langle X_1\rangle)\cap C_L^*(G\times X; \langle X_2\rangle)\big)=0$. This completes the proof of part (3).
\end{proof}

\begin{cor}\label{excision}
	Let $X_1, X_2$ be two $G$-invariant closed subspaces of $X$. If $(X, X_1, X_2)$ is coarsely excisive or if $X$ is a disjoint union of $X_1$ and $X_2$, then the inclusions of $C^*$-algebras induce an isomorphism 
	$$j_*: LK^G_*(X_1, X_1\cap X_2)\longrightarrow LK^G_*(X, X_2)$$
\end{cor}
\begin{proof} By Lemma \ref{ideal}, the inclusions of $C^*$-algebras
	induce isomorphisms on $K$-groups
	$$K_*\big(C^*_L(G\times X_1)\big)\cong K_*\big(C^*_L(G\times X; \langle X_1\rangle)\big)$$
	$$K_*\big(C^*_L(G\times X_1;\langle X_1\cap X_2\rangle)\big)\cong K_*\big(C^*_L(G\times X; \langle X_1\cap X_2\rangle)\big)$$
	(the second isomorphism holds because they are both isomorphic to $K_*\big(C^*_L(G\times (X_1\cap X_2))\big)$). Therefore by the exact sequence of $K$-groups associated to an ideal in a $C^*$-algebra, we get
	$$LK^G_*(X_1, X_1\cap X_2)\cong K_*\big(C^*_L(G\times X; \langle X_1\rangle)/C^*_L(G\times X;\langle X_1\cap X_2\rangle)\big).$$
	
	If $(X, X_1, X_2)$ is coarsely excisive, then by Lemma \ref{excisivetriad}, we get that $j_*: LK^G_*(X_1, X_1\cap X_2)\longrightarrow LK^G_*(X, X_2)$ is an isomorphism. 
	
	If $X$ is the disjoint union of $X_1$ and $X_2$, then Lemma \ref{excisivetriad} implies that $K_*\big(C^*_L(G\times X; \langle X_1\rangle)\big)\cong LK^G_*(X, X_2)$.  So we also have $j_*: LK^G_*(X_1, X_1\cap X_2)\longrightarrow LK^G_*(X, X_2)$ is an isomorphism. 
\end{proof}

    The next lemma, which is an analogue of \cite[Proposition 5.3]{BFJR}, gives a condition on the triad $(X, X_1, X_2)$ so that it is coarsely excisive. The main idea of its proof is similar to that of \cite[Proposition 5.3]{BFJR}. We carry out the detail here since  some detail is missing there and also for convenience.

\begin{lem}\label{coarseexcisive}
	Let $X_1$ and $X_2$ be two closed $G$-invariant subspaces of a  $G$-space $X$ so that $X=X_1\cup X_2$. Assume there exists a $G$-invariant open neighborhood $U$ of $X_1\cap X_2$ in $X$ so that its closure $\bar U$ in $X$ is $G$-equivariantly homeomorphic to $(X_1\cap X_2)\times[-1, 1]$, where $G$ acts trivially on $[-1, 1]$. Moreover, $U$ corresponds to $(X_1\cap X_2)\times(-1, 1)$, $U\cap X_1$ corresponds to $(X_1\cap X_2)\times [0,1)$ and $U\cap X_2$ corresponds to $(X_1\cap X_2)\times (-1, 0]$. Then the triad $(X, X_1, X_2)$ is coarsely excisive. 
\end{lem}
\begin{proof}
	Let $E\in\mathscr E^G_X$ and $K\subseteq X$ be compact. For $i=1, 2$, define 
	$$\big(X_i\times[0, \infty)\big)^{E, K}=\{(x, t)\in X\times[0, \infty)\ |\ \exists g\in G, x_i\in X_i\text{\ so\ that}\ g^{-1}x, g^{-1}x_i\in K, (x, x_i, t)\in E\}$$
    Let 
    $$Z_{i}=\big(X_i\times [0,\infty)\big)^{E, K}-X_i\times [0, \infty), i=1, 2$$ 
    For $0\le\epsilon\le 1$, let $U_\epsilon$ be the part of $X$  corresponding to $(X_1\cap X_2)\times [-\epsilon, \epsilon]$. Note that for each $\epsilon>0$, there exists $t_\epsilon>0$, so that for any $(x, t)\in Z_1\cup Z_2$, if $t\ge t_\epsilon$, then $x\in U_\epsilon$. Otherwise, we can find an $\epsilon_0>0$, so that there exist a sequence $(x_n, t_n)\in Z_1\cup Z_2$ with $t_n\rightarrow\infty$ and $x_n\notin U_{\epsilon_0}$. Without loss of generality, we may assume $x_n\in X_1,\forall n\in\mathbb N$. Then $(x_n, t_n)\in Z_2$. Therefore, there exist sequences $y_n$ in $X_2$ and $g_n\in G$ so that $(x_n, y_n, t_n)\in E$, and $g_n^{-1}x_n, g_n^{-1}y_n\in K$. By passing to a subsequence, we may assume $g_n^{-1}y_n\rightarrow\bar{y}$. Note that $\bar{y}\in X_2$ since $X_2$ is closed. Now let $V=X_2\cup (X_1\cap X_2)\times(-\epsilon_0, \epsilon_0)$. Then $V$ is a $G$-invariant open neighborhood of $\bar{y}$ in $X$. Therefore, by definition of $E$, there exist a $G_{\bar y}$-invariant open neighborhood $W\subseteq V$ of $\bar y$ in $X$ and $N>0$, so that $(a,b, t)\in E, t>N, b\in W$ implies $a\in V$. Now since $g_n^{-1}y_n\rightarrow \bar y, t_n\rightarrow\infty$, and $(g_n^{-1}x_n, g_n^{-1}y_n, t_n)\in E$, we see that $g_n^{-1}x_n$ is eventually in $V$. So $x_n$ is eventually in $V\cap X_1\subseteq U_{\epsilon_0}$, which contradicts to $x_n\notin U_{\epsilon_0}, \forall n$. 
	
	For any $t\ge t_{1/2}$, define 
	$$\epsilon(t)=\text{inf}\{\epsilon\ |\ (Z_1\cup Z_2)\cap X\times[t,\infty)\subseteq U_\epsilon\times [0, \infty)\}$$
	Also, for any $z\in X_1\cap X_2$, let $L_\epsilon (z)\subseteq U$ correspond to $\{z\}\times[-\epsilon, \epsilon]$. Now set 
	$$E'=\{(x, y, t)\ |\ t<t_{1/2}\ \text{or}\ x=y\
	\text{or}\ x\in X_1\cap X_2\ \text{and}\ y\in L_{\epsilon(t)}(x)\ \text{or}\ y\in X_1\cap X_2\ \text{and}\ x\in L_{\epsilon(t)}(y) \}$$
	$$K'=\pi(K\cap\bar{U})\cup\{p\}$$
	where $\pi: \bar U\rightarrow X_1\cap X_2$ is the projection, and $p\in X_1\cap X_2$ is an arbitrary point. 
	
	We first show $E'\in\mathscr E^G_X$. The only non-trivial part is to show $E'$ satisfies condition (1) in Definition \ref{ecc}. Let $x\in X$ and $V_x$ be a $G_x$-invariant neighborhood of $x$ in $X$. We may assume $x\in X_1$. 
	
	If $x\notin X_1\cap X_2$, then there exists $\epsilon_0>0$ so that $x\notin X_2\cup U_{\epsilon_0}$. Choose $t_{\epsilon_0}>0$ so that $(Z_1\cup Z_2)\cap X\times [t_{\epsilon_0}, \infty)\subseteq U_{\epsilon_0}\times[0, \infty)$. Now let $W_x=V_x\cap(X-X_2\cup U_{\epsilon_0})$ and $N=\max\{t_{\epsilon_0}, t_{1/2}\}$. Then for any $t>N$, we have $\epsilon(t)\le\epsilon_0$. Hence $L_{\epsilon(t)}(z)\subseteq U_{\epsilon_0}$ for any $z\in X_1\cap X_2, t>N$. This implies that if $(y, z, t)\in E', t>N$ and one of $y, z$ lies in $W_x$, then $y=z$ (otherwise we must have $y, z\in U_{\epsilon_0}$ by the definition of $E'$, which contradicts to $U_{\epsilon_0}\cap W_x=\emptyset$).
	
	If $x\in X_1\cap X_2$. Then we can choose a $G_x$-invariant open neighborhood $Q_x$ of $x$ in $X_1\cap X_2$ and $\epsilon_0>0$ so that
	$Q_x\times (-2\epsilon_0, 2\epsilon_0)\subseteq V_x$. Choose $t_{\epsilon_0}>0$ as before. Let $W_x=Q_x\times(-\epsilon_{0}, \epsilon_{0})$ and $N=\max\{t_{\epsilon_0},t_{1/2}\}$. It is not hard to see that if $(y, z, t)\in E', t>N$ and one of $y, z$ lies in $W_x$, then the other one must lie in $Q_x\times [-\epsilon_0, \epsilon_0]\subseteq V_x$. Altogether, we have shown that $E'\in\mathscr E^G_X$.
	
	Now let $g\in G, x\in X, t\ge 0$ so that there exist $x_1\in X_1, x_2\in X_2$ with  $(x, x_1, t), (x, x_2, t)\in E, g^{-1}x, g^{-1}x_1, g^{-1}x_2\in K$. If $x\in X_1\cap X_2$, then we choose $x_0=x$ and we have $(x, x_0, t)\in E', g^{-1}x_0\in K'$. If $t<t_{1/2}$, then we choose $x_0=gp$ and we have $(x, x_0, t)\in E', g^{-1}x_0\in K'$. If $t\ge t_{1/2}$ and $x\notin X_1\cap X_2$. We may assume $x\in X_1-X_2$. Then we have $(x, t)\in Z_2$. Since $t\ge t_{1/2}$, there exists $x_0\in X_1\cap X_2$ so that $x\in L_{\epsilon(t)}(x_0)$. Hence $(x, x_0, t)\in E'$, and $g^{-1}x_0=g^{-1}\pi(x)=\pi(g^{-1}x)\in K'$. Altogether, we have shown that $(X, X_1, X_2)$ is coarsely excisive. 
\end{proof}

\subsection{A vanishing result on $K$-groups}
    For a space $X$, let $\mathcal CX=X\times[0, 1]/X\times\{0\}$ be the cone over $X$ and $\pi: X\times[0, 1]\longrightarrow\mathcal CX$ be the quotient map. For $(x, s)\in X\times[0, 1]$, we denote $\pi(x, s)$ by $[x, s]$. If $G$ acts on $X$, then we endow $X\times[0, 1]$ and $\mathcal CX$ with the $G$-actions given by $g(x, s)=(gx, s)$ and  $g[x, s]=[gx, s]$ respectively. The next lemma on the vanishing of the $K$-theory of $C^*_{L, 0}(G\times\mathcal CX)$ as defined in Remark \ref{obs_loc_alg} will play an important role in proving the homotopy invariance of $LK^G_*$. 
    

\begin{lem}\label{conevanish} For any $G$-CW-complex $X$, the $K$-groups of $C^*_{L, 0}(G\times\mathcal CX)$ vanish. 
\end{lem}

\begin{proof}
	We are going to construct an Eilenberg swindle on $C^*_{L, 0}(G\times\mathcal CX)$. The idea is to push elements of $C^*_{L, 0}(G\times\mathcal CX)$ towards the cone point of $\mathcal CX$ along with a pushing in the $t$-direction towards infinity. 
	
	For any  $0\le s\le 1$ and $f\in C_{L,0}[G\times\mathcal CX]=C^*_{L,0}(G\times\mathcal CX)\cap C_{L}[G\times\mathcal CX]$, 
	define
	\begin{align*}
	H_s(f)(t) &=
	\begin{cases}
	0        & \text{if } 0\leq t<s \\
	f(t-s)        & \text{if } t\geq s
	\end{cases}
	\end{align*}
	$H_s(f)$ is still continuously controlled at infinity for each $s$. Also, by uniform continuity of $f$, $H_s(f)$ is continuous in $s$ (note that the usual contracting homotopy $C_s(f)(t)=f(st)$ is discontinuous in $s$).  Hence it gives a homotopy
	$$H_s: C^*_{L, 0}(G\times\mathcal CX)\longrightarrow C^*_{L, 0}(G\times \mathcal CX)$$
	from $id=H_0$ to $\mathcal S=H_1$.  So on $K$-groups, $\mathcal S_*$ is the identity morphism. $\mathcal S$ will be our pushing in the $t$-direction towards  infinity. However, it does not give a swindle directly since the continuous control condition at infinity will not be preserved after applying  $\mathcal S$ infinitly many times.

	We now define the pushing towards the cone point.
	For each $k=0, 1,2, \cdots$, define 
	$$\theta_k: \mathcal CX\longrightarrow \mathcal CX,\ \ \ \ \theta_k\big([x, s]\big)=[x, 2^{-\frac{1}{k+1}}s]$$
	We want to emphasize that it is important to have the pushing become smaller and smaller as $k\rightarrow\infty$ (as we have here $\ds\lim_{k\to\infty}2^{-\frac{1}{k+1}}s=s$) so that part (3) of the claim below holds and we can apply Lemma \ref{conjugate}. Pushing such as $\theta_k\big([x, s]\big)=[x, \frac{1}{2}s]$ or $\theta_k\big([x, s]\big)=[x, \frac{1}{k}s]$ will not work. 
	
	Choose an isometry 
	$$W_k: \mathcal H_{\mathcal CX}\longrightarrow\mathcal H_{\mathcal CX}$$
	so that it covers $\theta_k$, meaning $W_k(\mathcal H_{[x, s]})\subseteq\mathcal H_{[x, 2^{-\frac{1}{k+1}}s]}$. This is possible since $\theta_k$ is injective.  $W_k$ determines a unique $G$-invariant isometry
		$$V_k: \mathcal H_{G\times\mathcal CX}\longrightarrow\mathcal H_{G\times\mathcal CX}$$
	so that $V_k\big |_{\{e\}\times\mathcal CX}=W_k$ and $V_k(\mathcal H_{(g, [x, s])})\subseteq\mathcal H_{(g, [x, 2^{-\frac{1}{k+1}}s])}$.
	
	Define a family of isometries $U(t), t\ge 0$ from $\mathcal H_{G\times\mathcal CX}\oplus\mathcal H_{G\times\mathcal C X}$ to $\mathcal H_{G\times\mathcal CX}\oplus\mathcal H_{G\times\mathcal CX}$ by
	$$U(t)=R(t-k+1)(V_{k-1}\oplus V_{k})R^*(t-k+1),\ k-1\le t<k,\ k=1, 2, \cdots$$
	where $$R(t)=
	\begin{bmatrix}
	\ \ \cos(\frac{\pi t}{2}) & \sin(\frac{\pi t}{2})\\
	-\sin(\frac{\pi t}{2}) & \cos(\frac{\pi t}{2})
	\end{bmatrix}$$
	
	Let $\mathcal B$ be the $C^*$-algebra of all bounded functions from $[0, \infty)$ to $B(\mathcal H_{G\times\mathcal CX})$. 
	For $n=0, 1, 2\cdots$, and $\sigma\in M_2(\mathcal B)$, define
	$\mathcal S^n(\sigma)$ by
	\begin{align*}
	\mathcal S^n(\sigma)(t) &=
	\begin{cases}
	0        & \text{if } 0\leq t<n \\
	\sigma(t-n)        & \text{if } t\geq n
	\end{cases}
	\end{align*}
	and for  $f\in C_{L, 0}[G\times\mathcal CX]$, define
	$$f[n]=\mathcal S^n\big(U^n(f\oplus 0)(U^n)^*\big)$$ 
	We have $$f[n+1]=\mathcal S\big(\mathcal S^n(U)f[n]\mathcal S^n(U)^*)$$
	
    Choose a unitary isomorphism $\eta: \oplus_{\infty}\mathcal H\longrightarrow\mathcal H$, where $\oplus_{\infty}\mathcal H$ denotes the direct sum of countably and infinitly many copies of $\mathcal H$. It induces a unitary isomorphism
    $$\bar\eta: \oplus_{\infty}\mathcal H_{G\times\mathcal CX}\longrightarrow\mathcal H_{G\times\mathcal CX}$$
    $$\bar\eta\big((v_i)_{i=1}^\infty\big)(z)=\eta\big((v_i(z))_{i=1}^\infty\big), \ \ v_i\in\mathcal H_{G\times\mathcal CX},\  z\in G\times\mathcal CX$$
    For $t\ge 0$, define $\Phi(f)(t), \Psi(f)(t)$ and $\Omega(t)$ by
    $$\Phi(f)(t)=\bar{\eta}\circ(f[0](t)\oplus 0\oplus 0\oplus\cdots)\circ\bar{\eta}^*=\bar{\eta}\circ(f(t)\oplus 0\oplus 0\oplus 0\oplus\cdots)\circ\bar{\eta}^*$$
    $$\Psi(f)(t)=\bar\eta\circ\big(0\oplus 0\oplus f[1](t)\oplus f[2](t) \oplus\cdots\big)\circ\bar\eta^*$$
    $$\Omega(t)=\bar\eta\circ\big(\Omega_{ij}(t)\big)\circ\bar\eta^*$$
    where the $\mathbb N\times\mathbb N$ matrix $\big(\Omega_{ij}(t)\big)$ of operators represents the operator $\oplus_{\infty}\mathcal H_{G\times\mathcal CX}\longrightarrow\oplus_{\infty}\mathcal H_{G\times\mathcal CX}$, so that each $\Omega_{ij}(t)$ is a $2\times 2$ matrix of operators  with $\Omega_{ij}(t)=0$ if $i\ne j+1$ and $\Omega_{j+1, j}(t)=\mathcal S^{j-1}(U)(t)$.
    
    We shall prove the following claim:
	\vskip 10pt
	\noindent
	\underline{Claim:} 
   For any $f\in C_{L, 0}[G\times\mathcal CX]$, we have
	\begin{enumerate}
		\item $\Phi(f), \Psi(f)\in C_{L, 0}[G\times\mathcal CX]$, and  
		      $\Phi, \Psi$ extend as homomorphisms from $C^*_{L, 0}(G\times\mathcal CX)$ to itself;
		\item $\Omega\in\mathcal B$ is a partial isometry and $\big(\Phi(f)+\Psi(f)\big)\Omega^*\Omega=\Phi(f)+\Psi(f)$;
		\item $\Omega\cdot(\Phi(f)+\Psi(f))\in C_{L, 0}[G\times\mathcal  
		      CX]$;
		\item $\mathcal S\big(\Omega\cdot\big(\Phi(f)+\Psi(f)\big)\cdot\Omega^*\big)=\Psi(f)$. So in particular, $\Omega\cdot\big(\Phi(f)+\Psi(f)\big)\cdot\Omega^*\in C^*_{L, 0}(G\times\mathcal CX)$;
		\item on the level of $K$-groups, $\Phi_*$ is the identity map. 
	\end{enumerate}
    
    \vskip 10pt
    Assuming the claim, the lemma follows. To see this, note that since $\Phi(f)\Psi(g)=0$ for any $f, g\in C^*_{L, 0}(G\times\mathcal CX)$,  $\Phi+\Psi$ is also a homomorphism and $(\Phi+\Psi)_*=\Phi_*+\Psi_*$ on $K$-groups (\cite[Lemma 4.6.4]{HR}).
    (1)-(4) of the claim above enable us to apply Lemma \ref{conjugate} to Ad$(\Omega)\circ(\Phi+\Psi)$. Together with  $\mathcal S_*=id$, we get $\Phi_*+\Psi_*=\Psi_*$. Hence $\Phi_*=0$, which implies the $K$-groups of $C^*_{L, 0}(G\times\mathcal CX)$ vanish by (5) of the claim. 
    
    \vskip 10pt
    We now prove the claim.
    
    \vskip 10pt
    \noindent
    (1) Clearly, $\Phi(f)$ is still bounded and uniformly continuous. We have
    $$\Phi(f)(t)_{z', z}=\eta\circ\big(f(t)_{z',z}\oplus 0\oplus 0\oplus\cdots\big)\circ\eta^*,\ \forall z, z'\in G\times\mathcal CX$$
    Therefore $\Phi(f)(t)$ is locally compact and has the same support condition as $f$. Hence $\Phi(f)\in C_{L, 0}[G\times\mathcal CX]$.
    
    $\Psi(f)$ is easily seen to be a bounded function in $t$, and as we have seen in the proof of Lemma \ref{ideal}, each $f[n]$ is continuous in $t$ (although $U^n$ is discontinuous at $t=1,2,\cdots$). We compute that each entry of $f[n](t)$ is $0$ if $t\le n$ and is of the form
    $$\big(a(t)V^n_{k-n-1}+b(t)V^n_{k-n}\big)\circ f(t-n)\circ\big(c(t)(V_{k-n-1}^n)^*+d(t)(V_{k-n}^n)^*\big)$$
    if $t\ge n$, where $a(t), b(t), c(t), d(t)$ are products of sine and cosine functions in $t$ and $k$ is the integer so that $k-1\le t<k$. From this, we see that the family $f[n], n=0, 1, \cdots$, is uniformly equicontinuous in $t$. Thus $\Psi(f)$ is uniformly continuous in $t$. $\Psi(f)(t)$ is locally compact because
    for each $t$, there exists $N\in\mathbb N$, so that $f[n](t)=0, \forall n>N$. 
    
    For the support  of $\Psi(f)(t)$, we compute that 
    $$\Psi(f)(t)_{z', z}=\eta\circ\big(0\oplus 0\oplus f[1](t)_{z', z}\oplus f[2](t)_{z', z} \oplus\cdots\big)\circ\eta^*, \ \forall z, z'\in G\times\mathcal CX$$
    where $f[n](t)_{z', z}$ is the matrix obtained from taking the $(z', z)$-component of each entry of $f[n](t)$. Then $\Psi(f)(t)_{z', z}\ne 0$ implies $(z', z, t)\in E_f$, where
    $$E_f=\{(z', z, t)\ |\ z, z'\in G\times\mathcal CX, \exists n\in\mathbb N,\ i,j\in\{0, 1\}\ \text{s.t.}\ t\ge n\ \text{and}$$ $$\big(V^n_{k-n-i}\circ f(t-n)\circ(V^n_{k-n-j})^*\big)_{z', z}\ne 0, 
    \text{where}\ k\in\mathbb N\ \text{s.t.}\ k-1\le t<k\}$$
     Therefore, to show $\Psi(f)\in C_{L, 0}[G\times\mathcal CX]$, it remains to show $E_f$ satisfies the following properties:
    
    \begin{enumerate}[label=(\roman*)]
    	\item for each $t\ge 0$, there exists a finite subset $F_t\subseteq\mathcal CX$ (depending on $t$), so that if we have $\big((g', [x', s']), (g, [x, s]), t\big)\in E_f$, then $[g^{-1}x, s], [(g')^{-1}x', s']\in F_t$;
        \item there exist a finite subset $S\subseteq G$, a compact subset $K\subseteq\mathcal CX$, a countable $G$-invariant subset $\Lambda\subseteq\mathcal CX$ (all of them are independent of $t$), and $E\in\mathscr E^G_{\mathcal CX}$, so that
        $\big((g', [x', s']), (g, [x, s]), t\big)\in E_f$ implies $ [g^{-1}x, s], [(g')^{-1}x', s']\in K, g^{-1}g'\in S$, $[x, s], [x', s']\in\Lambda$ and
        $\big([x', s'], [x, s], t\big)\in E$.
    \end{enumerate}

    By definition, there exist a finite subset $S_1\subseteq G$, a compact subset $K_1\subseteq\mathcal CX$, a countable $G$-invariant subset $\Lambda_1\in\mathcal CX$,  and a set $E_1\in\mathscr E^G_{\mathcal CX}$ so that $f$ is $(S_1, K_1,\Lambda_1, E_1)$-controlled. Also, for each $t\ge 0$, there exists a finite subset $F'_{t}\subseteq\mathcal CX$ (depending on $t$) so that $f(t)_{(g', [x', s']), (g, [x, s])}\ne 0$ implies $[g^{-1}x, s], [(g')^{-1}x', s']\in F'_t$.
    
    For the compact set $K_1$, there exists a compact set $K_2\subseteq X\times[0, 1]$ so that $K_1\subseteq\pi(K_2)$ (this can be proved by using the fact that every compact subset of a $CW$-complex is contained in a finite subcomplex). Let $K_3$ be the projection of $K_2$ onto $X$. Now let
    $$S=S_1,\ K=\pi(K_3\times[0, 1])$$ 
    $$\Lambda=\bigcup_{k, n\in\mathbb N, k>n, i=0, 1}\{[x, s]\in\mathcal CX\ |\ [x, 2^{\frac{n}{k-n-i+1}}s]\in\Lambda_1\}$$
    $$E=\{\big([x', s'], [x, s], t\big)\ |\ [x', s'], [x, s]\in G\cdot K,\ \exists n\in\mathbb N,\ i,j\in\{0, 1\}\ \text{s.t.}\ t\ge n\ \text{and}$$
    $$ \big([x', 2^{\frac{n}{k-n-i+1}}s'], [x, 2^{\frac{n}{k-n-j+1}}s], t-n\big)\in E_1, \text{where}\ k\in\mathbb N\ \text{s.t.}\ k-1\le t< k \}$$
    Also for fixed $t\ge 1$, let $k\in\mathbb N$ so that $k-1\le t<k$ and
    $$F_t=\bigcup_{n\in\mathbb N, n\le t, i=0, 1}\{[x, s]\in\mathcal CX\ |\ [x, 2^{\frac{n}{k-n-i+1}}s]\in F'_{t-n}\},$$
    We show $(S, K, \Lambda, E)$ and $F_t$ are as desired. Certainly, $S$ and $F_t$ are finite, $K$ is compact, $\Lambda$ is countable and by Lemma \ref{conecontrol} below, $E$ is contained in a set in $\mathscr E^G_{\mathcal CX}$. 
    Now if
    $$\big((g', [x', s']), (g, [x, s]), t\big)\in E_f$$
    Then there exist $n\in\mathbb N, i, j\in\{0, 1\}$ so that $t\ge n$ and 
    $$\big(V^n_{k-n-i}\circ f(t-n)\circ(V^n_{k-n-j})^*\big)_{(g', [x', s']), (g, [x, s])}\ne 0$$
    where $k$ is the integer with $k-1\le t< k$.
    This implies
    $$f(t-n)_{(g', [x', 2^{\frac{n}{k-n-i+1}}s']), (g, [x, 2^{\frac{n}{k-n-j+1}}s])}\ne 0$$
    Hence
    $$[g^{-1}x, 2^{\frac{n}{k-n-j+1}}s],[(g')^{-1}x', 2^{\frac{n}{k-n-i+1}}s']\in F'_{t-n}\cap K_1,\ [x, 2^{\frac{n}{k-n-j+1}}s], [x', 2^{\frac{n}{k-n-i+1}}s']\in\Lambda_1 $$
    $$g^{-1}g'\in S_1,\ \big([x', 2^{\frac{n}{k-n-i+1}}s'], [x, 2^{\frac{n}{k-n-j+1}}s], t-n\big)\in E_1$$
    Therefore
    $$g^{-1}g'\in S,\ [g^{-1}x, s], [(g')^{-1}x', s']\in F_t\cap K,\ [x, s], [x', s']\in\Lambda,\ ([x', s'], [x, s], t)\in E,$$
    which shows $\Psi(f)\in C_{L, 0}[G\times\mathcal CX]$. 
    Finally, $\Phi, \Psi$ extend as homomorphisms from $C^*_{L, 0}(G\times\mathcal CX)$ to itself since they are homomorphisms on $C_{L, 0}[G\times\mathcal CX]$ which are norm preserving.
    
    \vskip 10pt
    \noindent
    (2) A direct computation shows that
    $$\Omega^*(t)\circ\Omega(t)=\bar\eta\circ(I\oplus\mathcal S(I)\oplus\mathcal S^2(I)\oplus\cdots)\circ\bar{\eta}^*$$
    where $I(t)=1, t\ge 0$ with $1$ the identity operator on $\mathcal H_{G\times\mathcal CX}\oplus\mathcal H_{G\times\mathcal CX}$. It follows that
    $\Omega^*(t)\circ\Omega(t)\circ\Omega^*(t)=\Omega^*(t)$, which implies $\Omega$ is a partial isometry, and $\big(\Phi(f)+\Psi(f)\big)\Omega^*\Omega=\Phi(f)+\Psi(f)$.
      
    \vskip 10pt
    \noindent 
    (3) We compute that
    $$\Omega(t)\circ\big(\Phi(f)(t)+\Psi(f)(t)\big)=\bar{\eta}\circ\big(\bar{\Omega}_{ij}(t)\big)\circ\bar{\eta}^*$$
    where each $\bar{\Omega}_{ij}(t)$ is a $2\times 2$ matrix of operators with $\bar{\Omega}_{ij}(t)=0$ if $i\ne j+1$ and $$\bar{\Omega}_{j+1, j}(t)=\mathcal S^{j-1}(U)(t)f[j-1](t)=\mathcal S^{j-1}\big(U^j(f\oplus 0)(U^{j-1})^*\big)(t),\ \ j=1, 2, \cdots$$
    From here, we see that $\Omega(t)\circ\big(\Phi(f)(t)+\Psi(f)(t)\big)$ is locally compact and uniformly continuous in $t$ (the argument is similar to that for $\Psi(f)$).
    
    To figure out the support condition, we compute that $\bar{\Omega}_{j+1, j}(t)=0$ if $t\le j-1$ and 
    $$\bar{\Omega}_{j+1, j}(t)=\big(a(t)V_{k-n-1}^{n+1}+b(t)V^{n+1}_{k-n}\big)\circ f(t-n)\circ\big(c(t)(V_{k-n-1}^n)^*+d(t)(V_{k-n}^n)^*\big)$$
    if $t\ge j-1$, where we have set $n=j-1$ and $k$ is the integer so that $k-1\le t< k$.
    Therefore, the situation here is similar to that of $\Psi(f)$. 
    The major part is to prove the following set 
    $$E=\{\big([x', s'], [x, s], t\big)\ |\ [x', s'], [x, s]\in G\cdot K,\ \exists n\in\mathbb Z,\ i,j\in\{0, 1\}\ \text{s.t.}\ t\ge n\ge 0\ \text{and}$$
    $$ \big([x', 2^{\frac{n+1}{k-n-i+1}}s'], [x, 2^{\frac{n}{k-n-j+1}}s], t-n\big)\in E_1, \text{where}\ k\in\mathbb N\ \text{s.t.}\ k-1\le t< k \}$$
    is contained in a set in $\mathscr E^G_{\mathcal CX}$ for any compact subset $K\subseteq\mathcal CX$ and $E_1\in\mathscr E^G_{\mathcal CX}$. But this follows from Lemma \ref{conecontrol} below.

    \vskip 10pt
    \noindent 
    (4) This follows from a direct calculation using $f[n+1]=\mathcal S\big(\mathcal S^n(U)f[n]\mathcal S^n(U)^*)$.
    \vskip 10pt
    \noindent 
    (5) Consider the isometry $\bar J: \mathcal H_{G\times\mathcal CX}\longrightarrow\oplus_{\infty}\mathcal H_{G\times\mathcal CX},\ \ \bar J(v)=(v, 0, \cdots)$. We then have $$\Phi(f)(t)=\bar\eta\circ\bar{J}\circ f(t)\circ\bar {J}^*\circ\bar{\eta}^*$$
    Let $V=\bar{\eta}\circ\bar{J}$, and define $\mathcal V(t)=V, t\ge 0$.
    Then $\mathcal V\in\mathcal B$ is an isometry and  $\Phi(f)=\text{Ad}(\mathcal V)(f)$. One checks that $\mathcal V\cdot f\in C^*_{L, 0}(G\times\mathcal CX)$ for any $f\in C^*_{L, 0}(G\times\mathcal CX)$. Therefore $\Phi_*=id$ on the level of $K$-groups by Lemma \ref{conjugate}.
\end{proof}

\begin{lem}\label{conecontrol} Let $X$ be a $G$-CW-complex and $\mathcal CX$ be the cone over $X$ with the obvious $G$-action. Then for any $E_1\in\mathscr E^G_{\mathcal CX}$ and compact subset $K\subseteq\mathcal CX$, the set 
$$E=\{\big([x', s'], [x, s], t\big)\ |\ [x', s'], [x, s]\in G\cdot K,\ \exists n\in\mathbb Z,\ i,j, l\in\{0, 1\}\ \text{s.t.}\ t\ge n\ge 0\ \text{and}$$
$$ \big([x', 2^{\frac{n+l}{k-n-i+1}}s'], [x, 2^{\frac{n}{k-n-j+1}}s], t-n\big)\in E_1, \text{where}\ k\in\mathbb N\ \text{s.t.}\ k-1\le t< k \}$$
is contained in a set in $\mathscr E^G_{\mathcal CX}$.
\end{lem}

\begin{proof}
	It is clear that $E$ is $G$-invariant. Therefore it suffices to show $E$ satisfies condition (1) of Definition \ref{ecc}. Choose a compact subset
	$L\subseteq X$ so that $K\subseteq\pi(L\times[0, 1])$. Let $[x_0, s_0]\in\mathcal CX$ and $U$ be a $G_{[x_0, s_0]}$-invariant open neighborhood of $[x_0, s_0]$ in $\mathcal CX$. We consider in two cases.
	
	\vskip 10pt
	\noindent
	Case 1: $s_0=0$. Then $G_{[x_0, s_0]}=G$ and $\pi^{-1}(U)\subseteq X\times[0, 1]$ is a $G$-invariant open neighborhood of $X\times\{0\}$.
	By compactness of $L$, there exists $s_1>0$ so that $L\times[0, s_1]\subseteq\pi^{-1}(U)$. This implies that $(G\cdot L)\times[0, s_1]\subseteq\pi^{-1}(U)$. Let $U_1=U\cap\pi\big(X\times[0, s_1)\big)$. Then $U_1$ is a $G_{[x_0, s_0]}$-invariant open neighborhood of $[x_0, s_0]$.  Therefore, since $E_1\in\mathscr E^G_{\mathcal CX}$, there exist $N_1>1$ and a $G_{[x_0, s_0]}$-invariant open neighborhood $V_1\subseteq U_1$ of $[x_0, s_0]$ so that if $([x', s'], [x, s], t)\in E_1$ and $t>N_1$, then one of $[x', s'], [x, s]$ lies in $V_1$ implies the other lies in $U_1$. Choose $s_2>0$ so that $(G\cdot L)\times[0, s_2]\subseteq\pi^{-1}(V_1)$ and $N_2>1$ so that $2^{-\frac{N_2}{N_1+2}}<s_1$.
	
	Now let $N=N_1+N_2$ and $V=V_1\cap\pi\big(X\times[0, 2^{-\frac{4N_2}{N_1+2}}s_2)\big)$. Then $V\subseteq U$ is a $G$-invariant open neighborhood of $[x_0, s_0]$ in $\mathcal CX$.
    Assume $\big([x', s'], [x, s], t\big)\in E,\ t>N$. We show if one of $[x, s], [x', s']$ lies in $V$, then the other lies in $U$. This will complete the proof for Case 1.   By definition, we have $[x', s'], [x, s]\in G\cdot K$ and there exist $n\in\mathbb Z,\ i,j, l\in\{0, 1\}$ so that $t\ge n\ge 0,
	\big([x', 2^{\frac{n+l}{k-n-i+1}}s'], [x, 2^{\frac{n}{k-n-j+1}}s], t-n\big)\in E_1, \text{where}\ k\in\mathbb N\ \text{s.t.}\ k-1\le t< k$. 
	\vskip 10pt
	Subcase 1: $t-n\le N_1$. Then $n>N_2$ and $k-n\le N_1+1$. Hence
	$$1\ge 2^{\frac{n+l}{k-n-i+1}}s'>2^{\frac{N_2}{N_1+2}}s',\ \  1\ge 2^{\frac{n}{k-n-j+1}}s>2^{\frac{N_2}{N_1+2}}s.$$
	Therefore $$s', s<2^{-\frac{N_2}{N_1+2}}<s_1.$$
	We also have $$[x', s'], [x, s]\in G\cdot K\subseteq G\cdot\pi(L\times[0, 1])=\pi(G\cdot L\times[0, 1]).$$
	Thus $$[x', s'], [x, s]\in\pi(G\cdot L\times[0, s_1))\subseteq U.$$
	
	\vskip 10pt
	Subcase 2: $t-n>N_1$. Then $k-n>N_1>1$. 
	
	If $[x, s]\in V$, we show $[x', s']\in U$. We may assume $\frac{n+l}{k-n-i+1}\le \frac{N_2}{N_1+2}$ (otherwise, the same argument as in Subcase 1 will imply $[x', s']\in U$). Then
	$$\frac{n}{k-n-j+1}\le\frac{n+l}{k-n-i+1}\cdot\frac{k-n-i+1}{k-n-j+1}\le\frac{N_2}{N_1+2}\cdot\frac{k-n+1}{k-n}\le\frac{2N_2}{N_1+2}.$$
	This together with $s<2^{-\frac{4N_2}{N_1+2}}s_2$ implies
	$$2^{\frac{n}{k-n-j+1}}s<2^{\frac{4N_2}{N_1+2}}s<s_2.$$
	Also $[x, s]\in G\cdot K\subseteq\pi(G\cdot L\times[0, 1])$.
	Hence $[x, 2^{\frac{n}{k-n-j+1}}s]\in\pi\big(G\cdot L\times[0, s_2]\big)\subseteq V_1$.  Thus $[x', 2^{\frac{n+l}{k-n-i+1}}s']\in U_1=U\cap\pi\big(X\times[0, s_1)\big)$. Therefore $s'\le 2^{\frac{n+l}{k-n-i+1}}s'<s_1$. Together with $[x', s']\in G\cdot K\subseteq\pi\big(G\cdot L\times[0, 1]\big)$, it follows that $[x', s']\in\pi\big(G\cdot L\times[0, s_1]\big)$. Hence $[x', s']\in U$.
	Similar argument shows that if $[x', s']\in V$, then $[x, s]\in U$.  
	
	
	\vskip 10pt
	\noindent
	Case 2: $s_0>0$. Then $G_{[x_0, s_0]}$=$G_{x_0}$. For any $r>0$ and $s\in[0, 1]$, let $B_{r}(s)$ be the open ball of radius $r$ and center $s$ in $[0, 1]$.
	By choosing a smaller open neighborhood of $[x_0, s_0]$, we may assume $U=\pi\big(W\times B_{r}(s_0)\big)$ for some $G_{x_0}$-invariant open neighborhood $W$ of $x_0$ in $X$ and $0<r<\frac{1}{8}s_0$. For any $s_0-r\le s \le 1$, let $U_s=\pi\big(W\times B_{\frac{r}{4}}(s)\big)$. There exist $N_s>0, 0<r_s<\frac{r}{4}$ and a $G_{x_0}$-invariant open neighborhood $W_{s}\subseteq W$ of $x_0$ in $X$, so that if $([x', s'], [x'', s''], t)\in E_1$ and $t>N_s$, then one of $[x', s'], [x'', s'']$ lies in $\pi\big(W_{s}\times B_{r_s}(s)\big)$
	implies the other lies in $U_s$. Since $[s_0-r, 1]$ is compact, there exist $s_1, s_2, \cdots, s_m\in[s_0-r, 1]$ so that $[s_0-r, 1]\subseteq\bigcup_{p=1}^mB_{r_{s_p}}(s_p)$. Let $\delta=\min\{r_{s_1}, \cdots, r_{s_m}\}$, $W_{x_0}=\bigcap_{p=1}^{m}W_{s_p}$. Choose $N_0>1$ so that $(\frac{2}{s_0})^{\frac{1}{N_0}}-1<\frac{1}{32}rs_0$ and $1-(\frac{2}{s_0})^{-\frac{1}{N_0}}<\frac{1}{32}rs_0$. Let $N_1=N_0+N_{s_1}+\cdots N_{s_m}$. Choose $N_2>0$ so that $2^{\frac{N_2}{N_1+2}}>\frac{2}{s_0}$. 
	
	Now let $V=\pi\big(W_{x_0}\times B_{\delta}(s_0)\big)$ and $N=N_1+N_2$. Assume $([x', s'],[x, s], t)\in E$, $t>N$ and $[x, s]\in V$, we show $[x', s']\in U$. By definition, we have $[x', s'], [x, s]\in G\cdot K$ and there exist $n\in\mathbb Z,\ i,j, l\in\{0, 1\}$ so that $t\ge n\ge 0$ and $
	\big([x', 2^{\frac{n+l}{k-n-i+1}}s'], [x, 2^{\frac{n}{k-n-j+1}}s], t-n\big)\in E_1, \text{where}\ k\in\mathbb N\ \text{s.t.}\ k-1\le t< k$. By definition of $V$, we have $s>s_0-\delta>s_0-\frac{1}{4}r>\frac{1}{2}s_0$. Hence 
	$$2^{\frac{n}{k-n-j+1}}\cdot\frac{s_0}{2}<2^{\frac{n}{k-n-j+1}}s\le 1$$
	$$2^{\frac{n}{k-n-j+1}}s\ge s>s_0-r$$
	So $2^{\frac{n}{k-n-j+1}}<\frac{2}{s_0}$ and $2^{\frac{n}{k-n-j+1}}s\in B_{r_{s_p}}(s_p)$ for some $1\le p\le m$. This implies $[x, 2^{\frac{n}{k-n-j+1}}s]\in\pi\big(W_{s_p}\times B_{r_{s_p}}(s_p)\big)$ and $t-n>N_1$ (otherwise $k-n\le N_1+1$, $n\ge N_2$ and $2^{\frac{n}{k-n-j+1}}\ge 2^{\frac{N_2}{N_1+2}}>\frac{2}{s_0}$).  Therefore $[x', 2^{\frac{n+l}{k-n-i+1}}s']\in U_{s_p}=\pi(W\times B_{\frac{r}{4}}(s_p))$. Hence $x'\in W$ and
	\begin{align*}
    |s'-s_0|
	&\le |s'-s|+|s-s_0|\\
	&<|2^{\frac{n+l}{k-n-i+1}}s'-2^{\frac{n+l}{k-n-i+1}}s|+\delta\\
	&<|2^{\frac{n+l}{k-n-i+1}}s'-s_p|+|s_p-2^{\frac{n}{k-n-j+1}}s|+|2^{\frac{n}{k-n-j+1}}s-2^{\frac{n+l}{k-n-i+1}}s|+\frac{1}{4}r\\
	&<\frac{1}{4}r+r_{s_p}+|2^{\frac{n}{k-n-j+1}}-2^{\frac{n+l}{k-n-i+1}}|+\frac{1}{4}r\\
	&<2^{\frac{n}{k-n-j+1}}\cdot\big|1-2^{\frac{n+l}{k-n-i+1}-\frac{n}{k-n-j+1}}\big|+\frac{3}{4}r\\
	&<\frac{2}{s_0}\cdot\big|1-2^{\frac{l}{k-n-i+1}}\big|+\frac{2}{s_0}\cdot 2^{\frac{l}{k-n-i+1}}\cdot\big|1-2^{\frac{n}{k-n-i+1}-\frac{n}{k-n-j+1}}\big|+\frac{3}{4}r\\
	&\le\frac{2}{s_0}\cdot(2^{\frac{1}{N_0}}-1)+\frac{2}{s_0}\cdot 2^{\frac{1}{N_0}}\cdot\big|1-2^{\frac{n(i-j)}{(k-n-j+1)(k-n-i+1)}}\big|+\frac{3}{4}r\\
	&\le\frac{2}{s_0}\cdot\frac{1}{32}rs_0+\frac{4}{s_0}\cdot\max\{(\frac{2}{s_0})^{\frac{1}{N_0}}-1,\ 1-(\frac{2}{s_0})^{-\frac{1}{N_0}}\}+\frac{3}{4}r\\
	&<\frac{1}{16}r+\frac{4}{s_0}\frac{1}{32}rs_0+\frac{3}{4}r\\
	&<r
	\end{align*}
    It follows that $[x', s']\in\pi(W\times B_r(s_0))=U$. Similar argument  shows that if $[x', s']\in V$, then $[x, s]\in U$. This completes the proof of Case 2, and therefore the proof of the Lemma.
\end{proof}

\section{Proof of the Main Theorem}\label{proof_of_the_main_theorem}

We are now ready to prove Theorem \ref{eqkhomology}.

\begin{proof}
	(1) Functoriality: let $\phi: (X, A)\rightarrow (Y, B)$ be a $G$-equivariant map between $G$-$CW$-pairs.	
	By Definition \ref{localalg}, we have 
	$$C_L^*(G\times X)=\ds\text{colim}_{\Lambda\subseteq X}C_L^*(G\times X; \Lambda)$$
	where the direct colimit is taking over all $G$-invariant countable subsets of $X$ under inclusions.
	
	Let $\Lambda\subseteq X$ be a $G$-invariant countable subset. For each $y\in Y$, let $\Lambda_y=\Lambda\cap\phi^{-1}(y)$. Since $\Lambda_y$ is countable, there exists an isometry 
	$$W_y: \mathcal H_{\Lambda_y}\longrightarrow \mathcal H_y$$
	This gives rise to an isometry 
	$$W_{\Lambda}=\oplus_{y\in Y} W_y: \mathcal H_\Lambda\longrightarrow \mathcal H_{Y}$$
	It uniquely determines a $G$-invariant isometry 
	$$V_{\Lambda}: \mathcal H_{G\times\Lambda}\longrightarrow \mathcal H_{G\times Y}$$
	so that $V_\Lambda(\mathcal H_{\{g\}\times\Lambda})\subseteq\mathcal H_{\{g\}\times Y}$ and its restriction to $\mathcal H_{\{e\}\times \Lambda}$ is equal to $W_\Lambda$. Note that $V_\Lambda$ covers $id_G\times\phi$.
	
	For any $f\in C_L[G\times X;\Lambda]$, define $\text{Ad}(V_\Lambda)(f)$ by
	$$\text{Ad}(V_\Lambda)(f)(t):=V_\Lambda\circ f(t)\circ V_\Lambda^*, \ t\geq 0$$
	We claim $\text{Ad}(V_\Lambda)(f)\in C_L[G\times Y]$. The only nontrivial part is to show there exists $E_Y\in\mathscr E^G_Y$ so that
	$\text{Ad}(V_\Lambda)(f)(t)_{(g', y'),(g, y)}\ne 0$ implies $(y', y, t)\in E_Y$. As $V_\Lambda$ covers $\phi$, if $\text{Ad}(V_\Lambda)(f)(t)_{(g', y'),(g, y)}\ne 0$, then there exist
	$x', x\in X$ with $f(t)_{(g', x'), (g, x)}\ne 0$ so that $(y', y, t)=(\phi(x'), \phi(x), t)$. Therefore, together with the $G$-cocompact support condition, it suffices to show for any $E_X\in\mathscr E^G_X$ and compact set $K\subseteq X$, the set
	$$E_Y=\{(\phi(x'), \phi(x), t)\ |(x', x, t)\in E_X, x', x\in G\cdot K\}\cup\{(y, y, t)\ |y\in Y,\ t\ge 0\}$$
	belongs to $\mathscr E^G_Y$. But this follows from part (3) of Lemma \ref{ecc_properties}.
    Using this fact, we also see that if in addition $f\in C_L[G\times X; \langle A\rangle]$, then $\text{Ad}(V_\Lambda)(f)\in C_L[G\times Y; \langle B\rangle]$.
	Therefore, we have a  $C^*$-algebra homomorphism
	$$\text{Ad}(V_\Lambda): C^*_L(G\times X; \Lambda)\longrightarrow C^*_L(G\times Y)$$
	which maps $C^*_L(G\times X; \langle A\rangle, \Lambda)$ to $C_L^*(G\times Y; \langle B\rangle)$, where $C^*_L(G\times X; \langle A\rangle, \Lambda)$ is the closure of $C_L[G\times X; \langle A\rangle, \Lambda]=C_L[G\times X; \langle A\rangle]\cap C_L[G\times X; \Lambda]$ in $C^*_L(G\times X)$.  Hence, we have an induced homorphism
	$$\overline{\text{Ad}(V_\Lambda)}: C^*_L(G\times X;\Lambda)/C^*_L(G\times X; \langle A\rangle,\Lambda)\longrightarrow C^*_L(G\times Y)/C^*_L(G\times Y; \langle B\rangle)$$
	Using Lemma \ref{conjugate}, it is not hard to show that on the level of $K$-groups, $\overline{\text{Ad}(V_\Lambda)}_*$ is independent of the choice of $V_\Lambda$.	
	As $C^*_L(G\times X;\langle A\rangle, \Lambda)=C^*_L(G\times X;\langle A\rangle)\cap C^*_L(G\times X; \Lambda)$ and
	$$C^*_L(G\times X;\Lambda)/C^*_L(G\times X;\langle A\rangle)\cap C^*_L(G\times X; \Lambda)\cong\big(C^*_L(G\times X;\Lambda)+C^*_L(G\times X; \langle A\rangle)\big)/C^*_L(G\times X; \langle A\rangle)$$ 
	we obtain a homomorphism (which is still denoted by $\overline{\text{Ad}(V_\Lambda)}$)
	$$\overline{\text{Ad}(V_\Lambda)}: \big(C^*_L(G\times X;\Lambda)+C^*_L(G\times X; \langle A\rangle)\big/C^*_L(G\times X; \langle A\rangle)\longrightarrow C^*_L(G\times Y)/C^*_L(G\times Y; \langle B\rangle)$$
	Since $\overline{\text{Ad}(V_\Lambda)}_*$ is independent of the choice of $V_\Lambda$, we have, for $\Lambda_1\subseteq\Lambda_2$
	$$(\overline{\text{Ad}(V_{\Lambda_2})}\circ j_{\Lambda_1\rightarrow\Lambda_2})_*=\overline{\text{Ad}(V_{\Lambda_1})}_*$$
	where $j_{\Lambda_1\rightarrow\Lambda_2}$ is the inclusion induced by the inclusion of $C^*_L(G\times X; \Lambda_1)$ in $C^*_L(G\times X; \Lambda_2)$. Thus $\overline{\text{Ad}(V_\Lambda)}_*$ induces a homomorphism
	$$\text{colim}_{\Lambda\subseteq X}(\overline{\text{Ad}(V_\Lambda)}_*): K_*\big(C_L^*(G\times X)/C^*_L(G\times X;\langle A\rangle)\big)\longrightarrow K_*\big(C^*_L(G\times Y)/C^*_L(G\times Y; \langle B\rangle)\big)$$
	We define $\phi_*: LK^G_*(X, A)\longrightarrow LK^G_*(Y, B)$ to be this homomorphism. It is not hard to check that $id_*=id$ and $(\phi\circ\psi)_*=\phi_*\circ\psi_*$.

	\vskip 10pt
	\noindent
	(2) Exactness: let $(X, A)$ be a $G$-CW-pair. Then $A$ is a closed subspace of $X$. Hence by Lemma \ref{ideal}, we have $LK^G_*(A)\cong K_*\big(C^*_L(G\times X; \langle A\rangle)\big)$. As $C^*_L(G\times X; \langle A\rangle)\subseteq C^*_L(G\times X)$ is an ideal, the exactness follows.  
	
	\vskip 10pt
	\noindent
	(3) $G$-homotopy invariance: by functoriality, it suffices to show the inclusion $\iota: (X, A)\rightarrow (X\times[0, 1], A\times[0, 1]),\ \iota(x)=(x, 0)$ induces an isomorphism $\iota_{*}: LK^G_*(X, A)\longrightarrow LK^G_*(X\times[0, 1], A\times[0, 1])$. Using exactness of $LK^G_*$, it suffices to show  $\iota_*: LK^G_*(X)\longrightarrow LK^G_*(X\times[0, 1])$ is an isomorphism.

	Let $Y=X\times [-1, 1]/X\times\{-1\}, Y_1=X\times[0, 1], Y_2=X\times [-1, 0]/X\times\{-1\}$. Then $Y_1\cup Y_2=Y,\ Y_1\cap Y_2=X\times\{0\}$. By Lemma \ref{coarseexcisive}, $(Y, Y_1, Y_2)$ is coarsely excisive. Hence by Corollary \ref{excision}, we have 
	$LK^G_*(Y_1, Y_1\cap Y_2)\cong LK^G_*(Y, Y_2)$. By exactness, we see that $\iota_*: LK^G_*(X)\longrightarrow LK^G_*(X\times[0, 1])$ is an isomorphism if and only if the inclusion of $Y_2$ into $Y$ induces an isomorphism $LK^G_*(Y_2)\cong LK^G_*(Y)$. 
	
	Now consider the evaluation maps
	$$ev_{Y_2}: C^*_L(G\times Y_2)\longrightarrow C^*(G\times Y_2)$$
	$$ev_{Y}: C^*_L(G\times Y)\longrightarrow C^*(G\times Y)$$
	By Lemma \ref{lemstr2}, the inclusion of $Y_2$ in $Y$ induces an isomorphism between the $K$-groups of $C^*(G\times Y_2)$ and $C^*(G\times Y)$. As $Y$ and $Y_2$ are $G$-homeomorphic to the cone $\mathcal CX=X\times[0, 1]/X\times\{0\}$, the $K$-groups of $C^*_{L, 0}(G\times Y_2)$ and $C^*_{L, 0}(G\times Y)$ vanish by Lemma \ref{conevanish}, which implies $LK^G_*(Y_2)\cong LK^G_*(Y)$.

	\vskip 10pt
	\noindent
	(4) Excision: if $A\cap B=\emptyset$, then by Corollary \ref{excision}, we get $j_*: LK^G_*(A, A\cap B)\longrightarrow LK^G_*(X, B)$ is an isomorphism. If $A\cap B\ne\emptyset$, then it is not clear whether $(X, A, B)$ is always coarsely excisive. We proceed as follows. Consider the $G$-CW subcomplexes $Y_1=A\times\{0\}\cup\big((A\cap B)\times[0, \frac{1}{2}]\big)$, $Y_2=B\times\{1\}\cup\big((A\cap B)\times[\frac{1}{2}, 1]\big)$, $Y=Y_1\cup Y_2$ of $X\times[0, 1]$. By Lemma \ref{coarseexcisive}, the triad $(Y, Y_1, Y_2)$ is coarsely excisive. Therefore the inclusion induces an isomorphism $LK^G_*(Y_1, Y_1\cap Y_2)\cong LK^G_*(Y, Y_2)$ by Corollary \ref{excision}. Let $Z=X\times\{0\}\cup(B\times[0, 1])$ and $r: Z\longrightarrow X\times\{0\}$ be the obvious retraction which is a $G$-homotopy equivalence. The restrictions of $r$ give rise to $G$-maps $r_1: (Y_1, Y_1\cap Y_2)\longrightarrow\big(A\times\{0\}, (A\cap B)\times\{0\}\big)$ and $r_2: (Y, Y_2)\longrightarrow(X\times\{0\}, B\times\{0\})$. It is clear that $r_1: Y_1\longrightarrow A\times\{0\},\ r_1: Y_1\cap Y_2\longrightarrow (A\cap B)\times\{0\}$ and $r_2: Y_2\longrightarrow B\times\{0\}$ are $G$-homotopy equivalences. We claim $r_2: Y\longrightarrow X\times\{0\}$ is also a $G$-homotopy equivalence. This follows from the fact that  $Y$ is a $G$-equivariant deformation retract of $Z$. Such a deformation retract can be obtained by exactly the same way as that of \cite[Proposition 0.16]{HA}. Therefore, $r_2: Y\longrightarrow X\times\{0\}$, as the composition of two $G$-homotopy equivalences, is also a $G$-homotopy equivalences. These facts together with the exactness and homotopy invariance of $LK^G_*$ imply that $(r_1)_*: LK^G_*(Y_1, Y_1\cap Y_2)\longrightarrow LK^G_*\big(A\times\{0\}, (A\cap B)\times\{0\}\big)$ and $(r_2)_*: LK^G_*(Y, Y_2)\longrightarrow LK^G_*(X\times\{0\}, B\times\{0\})$ are isomorphims, from which the excision follows. 
	
	\vskip 10pt
	\noindent
	(5) Additivity: first assume $X$ is an arbitrary $G$-CW complexes. Suppose $\{X_\beta: \beta\in J\}$ is a directed system of $G$-CW subcomplexes of $X$ directed by inclusions and assume $X=\bigcup_{\beta\in J}X_\beta$.  For any $f\in C_L[G\times X]$, since $f$ has $G$-compact support, we see that there exists some $X_\beta$ so that $f\in C_L[G\times X; X_\beta]$ (a compact subset of a CW-complex is contained in a finite subcomplex). This implies that $C^*_L(G\times X)=\text{colim}_{\beta\in J}C^*_L(G\times X; X_\beta)$. Together with the fact that each subcomplex $X_\beta$ is a closed subspace of $X$ and Lemma \ref{ideal}, we get $LK^G_*(X)=\text{colim}_{\beta\in J}LK^G_*(X_\beta)$. 
	
	Now suppose $X=\coprod_{\alpha\in I} X_\alpha$.  If $I$ is a finite set, then the additivity is a formal consequence of the exactness and excision of $LK^G_*$. The general case can then be obtained by writing $X$ as the directed union of those subcomplexes which are finite unions of the subcomplexes $X_\alpha$ and the fact that $LK^G_*$ commutes with directed colimit.

    \vskip 10 pt
	\noindent
	(6) Value at $G/H$: we first show that for any $f\in C_L[G\times G/H]$, there exists $t_f\geq 0$, so that for any $t>t_f$, if $x\ne x'$, then $f(t)_{(g',x'),(g, x)}=0$. Choose sets $S\subseteq G, K\subseteq G/H$ and $E\in\mathscr E^G_{G/H}$ as in Definition \ref{localalg} for $f$. For any $x\in G/H$, since $G/H$ is discrete, by definition of $E$, there exists $t_x>0$, so that for any $t>t_x$, $(x, x', t)\in E$ implies $x'=x$. Therefore for any $t>t_x$,  $f(t)_{(g',x'), (g, x)}\neq 0$ or $f(t)_{(g,x), (g', x')}\neq 0$ for some $g, g'\in G$ implies $x=x'$. Now since $K$ is compact and discrete, it is finite, say $K=\{x_1, \cdots, x_m\}$. 
	Let $t_f=\text{max}\{t_{x_1}, \cdots, t_{x_m}\}$. For any $t>t_f$, suppose $f(t)_{(g', x'), (g, x)}\neq 0$, then there exists $x_i$, so that $g^{-1}x=x_i$. By $G$-invariance of $f(t)$, we see that $f(t)_{(g^{-1}g', g^{-1}x'), (e, x_i)}\neq 0$. This implies $g^{-1}x'=x_i$. Therefore $x=x'$.
	
	Now for any $\alpha\geq 0$, let $C_L[G\times G/H; \alpha]$ be the $*$-algebra of all $f\in C_L[G\times G/H]$ so that  $t\geq\alpha, x\neq x'$ implies $f(t)_{(g', x'), (g, x)}=0$ for any $g, g'\in G$. According to the previous paragraph, we have $C_L[G\times G/H]=\bigcup_{\alpha\ge 0}C_L[G\times G/H;\alpha]$. Therefore $C^*_L(G\times G/H)=\text{colim}_{\alpha\ge 0}C^*_L(G\times G/H; \alpha)$, where $C^*_L(G\times G/H; \alpha)$ is the closure of $C_L[G\times G/H; \alpha]$, and the colimit is taking under the inclusion $C^*_L(G\times G/H; \alpha)\subseteq C^*_L(G\times G/H; \beta)$ for any $\alpha\le\beta$. Hence $K_i\big(C^*_L(G\times G/H)\big)\cong \text{colim}_{\alpha\ge 0}K_i\big(C^*_L(G\times G/H; \alpha)\big), i=1, 2$. It is not hard to see that every inclusion $C^*_L(G\times G/H; \alpha)\hookrightarrow C^*_L(G\times G/H; \beta)$ is a homotopy equivalence with a homotopy inverse given by a shift. Therefore $K_i\big(C^*_L(G\times G/H)\big)\cong K_i\big(C^*_L(G\times G/H; 0)\big), i=1, 2$. By Lemma \ref{lemmorecontrol}, it suffices to show $K_i\big(C^*_L(G\times G/H; 0)\big)\cong K_i(C^*(G\times G/H; 0)), i=1, 2$.

    There is an evaluation map $r: C_L[G\times G/H; 0]\rightarrow C[G\times G/H; 0] $ by $r(f)=f(0)$, which is clearly a surjective $*$-algebra homomorphism that is norm-non-increasing. Therefore it extends to a surjective homomorphism $\bar{r}: C_L^*(G\times G/H;0)\rightarrow C^*(G\times G/H; 0)$. Let $C_{L, 0}^*(G\times G/H; 0)$ denote the kernel of $\bar r$. We show its $K$-groups vanish, and this will complete the proof.
   

	The method of proof is very similar to that of Lemma \ref{conevanish}. We will construct an Eilenberg swindle on $C_{L, 0}^*(G\times G/H; 0)$. The construction of the swindle here is much simpler as elements of $C_{L, 0}^*(G\times G/H; 0)$ have their supports, in the $G/H$ direction, on the diagonal of $G/H\times G/H$. Hence the continuous control condition  will always be satisfied when we shift $f\in C_{L, 0}^*(G\times G/H; 0)$ in $t$ towards infinity even infinitly many times (without performing any pushing in the $G/H$-direction). We will outline the main construcion, and omit most of the details as they are very similar to and much simpler than that in the proof of \ref{conevanish}.
	
	
	Define
	$$S: C_{L, 0}^*(G\times G/H; 0)\rightarrow C_{L, 0}^*(G\times G/H; 0)$$
	\begin{align*}
	S(f)(t) &=
	\begin{cases}
	0        & \text{if } 0\leq t \leq 1 \\
	f(t-1)        & \text{if } t\geq 1
	\end{cases}
	\end{align*}
	Then $S$ is a $C^*$-algebra homomorphism which is homotopic to the identity. 
	
	For the separable infinite dimensional Hilbert space $\mathcal H$, we choose a unitary isomorphism $U: \mathcal H\rightarrow\oplus_{\infty}\mathcal H$.
	For each $T\in B(\mathcal H)$, define  
	$$\sigma_i(T):=U^*\circ (0\oplus\cdots\oplus 0\oplus T\oplus 0\oplus\cdots)\circ U\in B(\mathcal H),$$
	where $T$ is in the $i$-th place, $i=1, 2, \cdots$. 
    Each $\sigma_i$ induces, for any space $X$,  a map 
    $$\Phi_i^X: B(\mathcal H_{X})\rightarrow B(\mathcal H_{X})$$
    $$\Phi^X_i(T)_{x', x}:=\sigma_i(T_{x', x}),\ \ \forall x, x'\in X.$$ 
    $\Phi_i^X$ preserves supports of operators and  maps locally compact operators to locally compact operators. 
    
    For any $f\in C_{L, 0}^*(G\times G/H; 0)$, define $$\Phi_i(f)(t):=\Phi_i^{G\times G/H}\big(S^{i-1}(f)(t)\big),\ \ t\geq 0,\ \ i=1, 2, \cdots.$$ 
    Then $\Phi_i(f)\in C_{L, 0}^*(G\times G/H; 0)$ and 
    $$\Phi_i: C_{L, 0}^*(G\times G/H; 0)\rightarrow C_{L, 0}^*(G\times G/H; 0)$$ is a $C^*$-algebra homomorphism.  $\Phi_i, \Phi_j, i\neq j$, are orthogonal to each other, i.e. $\Phi_i(f)\Phi_j(g)=0, \forall f, g\in C_{L, 0}^*(G\times G/H; 0)$. Now let 
    $$\Phi=\Phi_1,\ \ \Psi=\sum_{i=2}^{\infty}\Phi_i$$ 
    Note that for each $t\ge 0$, there exists $n\in\mathbb N$, so that $S^i(f)(t)=0$ for all $i\geq n$. Therefore $\Psi(f)$ is  still locally compact and thus an element of $C_{L, 0}^*(G\times G/H; 0)$. $\Psi$ and $\Phi+\Psi$ are $C^*$-algebra homomorphisms since $\Phi_i\perp\Phi_j, i\neq j$. Now similar and simpler arguments as in the proof of Lemma \ref{conevanish} show that, on the level of $K$-groups, $\Phi_*=id$, and $(\Phi+\Psi)_*=\Psi_*$. These together with the fact that  $(\Phi+\Psi)_*=\Phi_*+\Psi_*$ (since $\Phi\perp\Psi$, see \cite[Lemma 4.6.4]{HR}), imply the $K$-groups of $C_{L, 0}^*(G\times G/H; 0)$ vanish, and thus completes the proof.
\end{proof}

\section{Induction}\label{induction}

Let $H$ be a subgroup of $G$ and $X$ be a left $H$-space. Then the \textit{induction} of the $H$-space $X$, denoted by $G\times_HX$ is the quotient space of the space $G\times X$ by the right $H$-action $(g, x)h=(gh, h^{-1}x)$. The image of $(g, x)\in G\times X$ in $G\times_HX$ under the quotient map will be denoted by $[g, x]$. $G\times_H X$ is a left $G$-space by the action $g'[g, x]=[g'g, x]$. It is easily verified that $LK_*^G(G\times_H-)$ is an $H$-equivariant homology theory. The main result of this section is the following theorem:

\begin{thm}\label{induction_str}Let $G$ be a countable and discrete group and $H$ be a subgroup of $G$. Then there is a natural isomorphism between $H$-equivariant homology theories 
$$\eta: LK^H_*(-)\rightarrow LK^G_*(G\times_H-).$$
More precisely, for any $H$-CW-pair $(X, A)$, there is an induction isomorphism
$$\eta_{X, A}: LK^H_*(X, A)\rightarrow LK_*^G(G\times_HX, G\times_H A),$$
which is compatible with the boundary map and  natural in $(X, A)$: for any $H$-equivariant continuous map  $\phi: (X, A)\rightarrow (Y, B)$, let $\tilde{\phi}: (G\times_HX, G\times_H A)\rightarrow (G\times_HY, G\times_H B)$ be the induced map defined by $\tilde{\phi}([g, x])=[g, \phi(x)]$, then the diagram 
$$
	\xymatrix{ 
			LK^H_*(X, A)\ar[r]^{\eta_{X, A}\ \ \ \ \ \ \ \ }\ar[d]^{\phi_*} & LK^G_*(G\times_HX, G\times_HA)\ar[d]^{\tilde{\phi}_*}\\ 
		    LK^H_*(Y, B)\ar[r]^{\eta_{Y, B}\ \ \ \ \ \ \ \ } & LK^G_*(G\times_HY, G\times_HB)
		}
	$$  
commutes.
\end{thm}
\begin{proof} We first construct a map $\Omega: C[H\times X]\rightarrow C[G\times G\times_HX]$.  For any $g\in G$, let
$X_g=\{[g, x]\in G\times_H X\ |\ x\in X\}$, and $e\in G$ be the identity element. We view $H\times X$ as a subset of $G\times X_e$ via the map $(h, x)\mapsto (h, [e, x])$. Then we have an inclusion $\iota: \mathcal H_{H\times X}\rightarrow\mathcal H_{G\times X_e}$ and a projection $\pi: \mathcal H_{G\times X_e}\rightarrow\mathcal H_{H\times X}$ of Hilbert spaces. For any $T\in C[H\times X]$, define 
$$T^e=\iota\circ T\circ\pi: \mathcal H_{G\times X_e}\rightarrow \mathcal H_{G\times X_e}.$$
In terms of components, we have
$$T^e_{(g', [e, x']), (g, [e, x])}=
\begin{cases}
T_{(g', x'),(g, x)}\ \ \textnormal{if}\ g', g\in H\\
0\ \ \ \ \ \ \ \ \ \ \ \ \ \ \textnormal{otherwise}
\end{cases}$$
Now for any $g\in G$, there is the bijective map $G\times X_e\rightarrow G\times X_g$ defined by $(r, [e, x])\mapsto(gr, [g, x])$. It induces a unitary isomorphism $l_g: \mathcal H_{G\times X_e}\rightarrow\mathcal H_{G\times X_g}$. Define $$T^g=l_g\circ T^e\circ l_g^{-1}:\mathcal H_{G\times X_g}\rightarrow\mathcal H_{G\times X_g}.$$
In terms of components, we have 
$$T^{g}_{(a', [g, x']), (a, [g, x])}=T^e_{(g^{-1}a', [e, x']), (g^{-1}a, [e, x])},\ a, a'\in G, x, x'\in X.$$
Notice if $g, s\in G$ so that  $X_{g}=X_{s}$, then $T^{g}=T^{s}$. To see this, since $X_{g}=X_{s}$, there exists $h\in H$ so that $gh=s$. Then
\begin{align*}
T^{s}_{(a', [g, x']), (a, [g, x])}&=T^{s}_{(a', [sh^{-1}, x']), (a, [sh^{-1}, x])}\\
&=T^{s}_{(a', [s, h^{-1}x']), (a, [s, h^{-1}x])}\\
&=T^{e}_{(s^{-1}a', [e, h^{-1}x']),(s^{-1}a, [e, h^{-1}x])}\\
&=T^{e}_{(h^{-1}g^{-1}a', [e, h^{-1}x']),(h^{-1}g^{-1}a, [e, h^{-1}x])}
\end{align*}
The above operator is non-zero only if both  $g^{-1}a', g^{-1}a\in H$. In this case, we have
\begin{align*}
T^{s}_{(a', [g, x']), (a, [g, x])}&=T^{e}_{(h^{-1}g^{-1}a', [e, h^{-1}x']),(h^{-1}g^{-1}a, [e, h^{-1}x])}\\
&=T_{(h^{-1}g^{-1}a', h^{-1}x'),(h^{-1}g^{-1}a, h^{-1}x)}\\
&=T_{(g^{-1}a', x'),(g^{-1}a, x)}\\
&=T^{g}_{(a', [g, x']), (a, [g, x])}
\end{align*}
Now choose a complete representatives $\{g_i\in G\ |\ i\in I \}$ for the left coset space $G/H$. Then $\mathcal H_{G\times G\times_H X}=\ds\bigoplus_{i\in I}\mathcal H_{G\times X_{g_i}}$. We define $\Omega: C[H\times X]\rightarrow C[G\times G\times_HX]$ by
$$\Omega(T)=\ds\bigoplus_{i\in I}T^{g_i}: \mathcal H_{G\times G\times_H X}\rightarrow \mathcal H_{G\times G\times_H X}.$$
$\Omega(T)$ is independent of the choices of $g_i, i\in I$.
We check $\Omega(T)$ is indeed an element of $C[G\times G\times_HX]$.
\begin{enumerate}
\item $\Omega(T)$ is bounded:  since it is easily seen that $\|\Omega(T)\|=\|T\|$;
\item $\Omega(T)$ is $G$-invariant: $\forall (g, [g_i, x]), (g', [g_i, x'])\in G\times G\times_H X$ and $r\in G$, 
there exist unique $g_j, j\in I$ and $h\in H$ so that $rg_i=g_jh$, then
\begin{align*}
\Omega(T)_{(rg', [rg_i, x']),(rg, [rg_i, x])}&=\Omega(T)_{(rg', [g_jh, x']),(rg, [g_jh, x])}\\
&=\Omega(T)_{(rg', [g_j, hx']),(rg, [g_j, hx])}\\
&=T^{g_j}_{(rg', [g_j, hx']),(rg, [g_j, hx])}\\
&=T^{e}_{(g_j^{-1}rg', [e, hx']),(g_j^{-1}rg, [e, hx])}\\
&=T^{e}_{(hg_i^{-1}g', [e, hx']),(hg_i^{-1}g, [e, hx])}\\
&=T^e_{(g_i^{-1}g', [e, x']),(g_i^{-1}g, [e, x])}\\
&=T^{g_i}{(g', [g_i, x']),(g, [g_i, x])}\\
&=\Omega(T)_{(g', [g_i, x']),(g, [g_i, x])}
\end{align*}
which shows $\Omega(T)$ is $G$-invariant;
\item $\Omega(T)$ is locally compact: obvious;
\item $\Omega(T)$ is $G$-finite and has finite propagation in the $G$-direction: there exist a finite subset $F\subseteq X$ and a finite subset $S\subseteq H$  so that if $T_{(h', x'), (h, x)}\ne 0$, then $h^{-1}x, (h')^{-1}x'\in F$ and $h^{-1}h'\in S$. Let $$F_e=\{[e, x]\in G\times_HX\ |\ x\in F\},$$
which is a finite subset of $G\times_HX$. If 
$$\Omega(T)_{(g', [g_i, x']),(g, [g_i, x])}=T^e_{(g_i^{-1}g', [e, x']),(g_i^{-1}g, [e, x])}\ne 0,$$ 
then $g_i^{-1}g', g_i^{-1}g\in H$ and $T^e_{(g_i^{-1}g', [e, x']),(g_i^{-1}g, [e, x])}=T_{(g_i^{-1}g', x'),(g_i^{-1}g, x)}\ne 0$. Hence $g^{-1}g_ig_i^{-1}g'=g^{-1}g'\in S$ and  
$(g')^{-1}g_ix', g^{-1}g_ix\in F$, which implies 
$$[(g')^{-1}g_i, x']=[e, (g')^{-1}g_ix'],\ [g^{-1}g_i, x]=[e, g^{-1}g_ix]\in F_e.$$
This shows $\Omega(T)$ is $G$-finite and has finite propagation in the $G$-direction.

\end{enumerate}

Altogether, we have shown that $\Omega(T)\in C[G\times G\times_HX]$. 
Now for any $f\in C_L[H\times X]$, define $\Omega(f)(t)=\Omega(f(t)), t\ge 0$ (we do  not distinguish the notations as it should be clear from the context). We check $\Omega(f)\in C_L[G\times G\times_H X]$. Clearly, $\Omega(f)$ is bounded and uniformly continuous in $t$. By definition of $f$, there exist a finite subset $S\subseteq H$, a compact subset $K\subseteq X$, an $H$-invariant countable subset $\Lambda\subseteq X$, and $E\in\mathscr E^H_X$ so that $f$ is $(S, K, \Lambda, E)$-controlled. For any $t\ge 0$, if  
$$\Omega(f(t))_{(g', [g_i, x']),(g, [g_i, x])}=f(t)^e_{(g_i^{-1}g', [e, x']),(g_i^{-1}g, [e, x])}\ne 0,$$
then $g_i^{-1}g', g_i^{-1}g\in H$ and $f(t)^e_{(g_i^{-1}g', [e, x']),(g_i^{-1}g, [e, x])}=f(t)_{(g_i^{-1}g', x'),(g_i^{-1}g, x)}\ne 0$. Therefore $g^{-1}g_ig_i^{-1}g'=g^{-1}g'\in S$, $x, x'\in\Lambda, g^{-1}g_ix, (g')^{-1}g_ix'\in K, (x, x', t)\in E$. Define
$$S_1=S$$
$$\Lambda_1=\{[g_i, x]\in G\times_H X\ |\ i\in I, x\in\Lambda\},$$
$$K_1=\{[e, x]\in G\times_HX\ |\ x\in K\},$$
$$E_1=\{([g_i, x], [g_i, x'], t)\ |\ i\in I, (x, x', t)\in E\}$$
Then $\Omega(f)$ is $(S_1, \Lambda_1, K_1, E_1)$-controlled with $S_1$ finite, $\Lambda_1$ countable and $G$-invariant  and $K_1$ compact. It remains to show $E_1\in\mathscr E^G_{G\times_H X}$. One easily checks $E_1$ is symmetric, $G$-invariant and contains the diagonal. Now fix any $[g_i, x]\in G\times_HX$. Let $U\subseteq G\times_HX$ be a $G_{[g_i, x]}$-invariant open  neighborhood  of $[g_i, x]$. Define 
$$U_i=\{y\in X\ |\ [g_i, y]\in U\}.$$
Then $U_i$ is an open neighborhood of $x\in X$. We claim $U_i$ is $H_x$-invariant. Indeed, for any $h\in H_x, y\in U_i$, 
$[g_i, hy]=[g_ih, y]=[gg_i, y]$, where $g=g_ihg_i^{-1}$. It is clear that $g\in G_{[g_i, x]}$, so $[gg_i, y]=g[g_i, y]\in U$ since $[g_i, y]\in U$ and $U$ is $G_{[g_i, x]}$-invariant. Hence $[g_i, hy]\in U$ which implies $hy\in U_i$. Therefore $U_i$ is $H_x$-invariant. Therefore since $E\in\mathcal E^H_X$, there exists $N>0$ and an $H_x$-invariant open neighborhood $V_i\subseteq U_i$ with the property that if $t>N$ and $(x, x', t)\in E$, then one of $x, x'\in V_i$ implies the other lies in $U_i$. Now let
$$V=\{[g_i, y]\in G\times_HX\ |\ y\in V_i\}.$$
Then $V\subseteq U$ is an open neighborhood of $[g_i, x]$. It is $G_{[g_i, x]}$-invariant since $\forall [g_i, y]\in V$ and $g\in G_{[g_i, x]}$, there exists $h\in H_x$, so that $g=g_ihg_i^{-1}$, hence 
$$g[g_i, y]=[gg_i, y]=[g_ih, y]=[g_i, hy]\in V.$$
Now if $t>N$, $([g_j, y], [g_j, y'], t)\in E_1$ and $[g_j, y]\in V$, then $i=j, y\in V_i$ and $(y, y', t)\in E$, so $y'\in U_i$, thus $[g_j, y']=[g_i, y']\in U$. This shows $E_1\in\mathscr E^G_{G\times_HX}$. Hence we obtain a map
$$\Omega: C_L[H\times X]\rightarrow C_L[G\times G\times_HX]$$
which is easily seen to be a homomorphism of $*$-algebras ($\Omega$ is essentially an inclusion map). 

Next we show if $f\in C_L[H\times X; \langle A\rangle]$, where $A\subseteq X$ is an $H$-invariant subcomplex, then $\Omega(f)\in C_L[G\times G\times_HX; \langle G\times_H A\rangle]$. Choose $E_2\in \mathscr E^H_X$ and compact $K_2\subseteq X$ with the following property: if $f(t)_{(h', x'), (h, x)}\neq 0$, then $h^{-1}x, (h')^{-1}x'\in K_2$ and there exist $a, a'\in A$ so that $h^{-1}a, (h')^{-1}a'\in K_2$ and $(a, x, t), (a', x', t)\in E_2$.
Define
$$K_3=\{[e, x]\in G\times_HX\ |\ x\in K_2\},$$
$$E_3=\{([g_i, a], [g_i, x], t)\ |\ i\in I, (a, x, t)\in E_2\}.$$
Then $K_3$ is compact and $E_3\in\mathscr E^G_{G\times_HX}$ (its proof is exactly the same as the one for $E_1$). If  
$$\Omega(f(t))_{(g', [g_i, x']),(g, [g_i, x])}=f(t)^e_{(g_i^{-1}g', [e, x']),(g_i^{-1}g, [e, x])}\ne 0,$$
then $g_i^{-1}g', g_i^{-1}g\in H$ and $f(t)^e_{(g_i^{-1}g', [e, x']),(g_i^{-1}g, [e, x])}=f(t)_{(g_i^{-1}g', x'),(g_i^{-1}g, x)}\ne 0$. Then $g^{-1}g_ix, (g')^{-1}g_ix'\in K_2$ and there exist $a, a'\in A$ so that $g^{-1}g_ia, (g')^{-1}g_ia'\in K_2$ and $(a, x, t), (a', x', t')\in E_2$.  Then
$$g^{-1}[g_i, x]=[e, g^{-1}g_ix], (g')^{-1}[g_i, x']=[e, (g')^{-1}g_ix']\in K_3,$$
$$g^{-1}[g_i, a]=[e, g^{-1}g_ia], (g')^{-1}[g_i, a']=[e, (g')^{-1}g_ia']\in K_3,$$
$$([g_i, a], [g_i, x], t), ([g_i, a'], [g_i, x'], t)\in E_3.$$
These together with  $[g_i, a], [g_i, a']\in G\times_HA$ completes the proof that $\Omega(f)\in C_L[G\times G\times_HX, \langle G\times_H A\rangle]$ whenever $f\in C_L[H\times X; \langle A\rangle]$.

Now as $\Omega$ preserves norms, it extends to a homomorphism of $C^*$-algebras
$$\Omega: C_L^*(H\times X)\rightarrow C_L^*(G\times G\times_HX),$$
which maps $C_L^*(H\times X; \langle A\rangle)$ to $C_L^*(G\times G\times_HX; \langle G\times_HA\rangle)$. Hence it induces a map on the quotient $C^*$-algebras
$$\bar{\Omega}: C_L^*(H\times X)/C_L^*(H\times X; \langle A\rangle)\rightarrow C_L^*(G\times G\times_HX)/C_L^*(G\times G\times_HX; \langle G\times_HA\rangle).$$
We now define the induction map 
$$\eta_{X, A}: LK^H_*(X, A)\rightarrow LK_*^G(G\times_HX, G\times_H A)$$
to be the map on $K$-groups induced by $\bar{\Omega}$. $\eta_{X, A}$ is compatible with the boundary map since the boundary map in the six-term cyclic exact sequence is natural. It can be easily verified that $\eta_{X, A}$ is natural in $(X, A)$ as the construction is canonical.

Finally, to show $\eta_{X, A}$ is an isomorphism, it suffices to show it is an isomorphism for $(X, A)=(H/I, \emptyset)$, where $I<H$ is a subgroup, since $\eta$ is a natural transformation between two homology theories. As 
$LK^H_*(H/I)\cong C_r^*(I)$, and $LK^G_*(G\times_HH/I)\cong LK^G_*(G/I)\cong C_r^*(I)$, $\eta_{H/I}: LK^H_*(H/I)\rightarrow LK^G_*(G\times_HH/I)$ is expected to be an isomorphism. A rigorous argument is as follows. As we have shown in the proof of Theorem \ref{eqkhomology}, the inclusion maps 
$$C^*_L(H\times H/I; 0)\rightarrow C^*_L(H\times H/I),$$
$$C^*_L(G\times G\times_HH/I; 0)\rightarrow C^*_L(G\times G\times_H H/I)$$
induce isomorphisms on $K$-groups, hence it suffices to show $\Omega$, viewed as a map from $C^*_L(H\times H/I; 0)$ to $C^*_L(G\times G\times_HH/I; 0)$ induces an isomorphism on $K$-groups. As we have also shown in the proof of Theorem \ref{eqkhomology}, the evaluation maps
$$C^*_L(H\times H/I; 0)\rightarrow C^*(H\times H/I; 0)$$
$$C^*_L(G\times G\times_HH/I; 0)\rightarrow C^*(G\times G\times_HH/I; 0)$$
induce isomorphisms on $K$-groups.  As they also commute with $\Omega$, it suffices to show $\Omega$, viewed as a map from $C^*(H\times H/I; 0)$ to $C^*(G\times G\times_HH/I; 0)$, induces an isomorphism on $K$-groups. Now by Lemma \ref{lemmorecontrol}, it suffices to show the inclusion map 
$$C^*(I\times I\backslash H)\rightarrow C^*(I\times I\backslash G)$$
induces an isomorphism on $K$-groups, which is true by Lemma \ref{lemstr2}. Altogether, we have completed the proof of the theorem.
\end{proof}

\section{Transitivity Principle}\label{transitivity}

In this final section, we prove Theorem \ref{transitivity_principle}, namely the \textit{transitivity principle}. The strategy of the proof is exactly the same as that of \cite[Theorem 65]{LR} and the main ingredient is the induction isomorphism proved in the previous section.  We carry out the detail for convenience.

\begin{proof}
First note that, by the fixed point set characterization of classifying spaces, $E_{\mathcal F}G\times E_{\mathcal F'}G$ with the diagonal $G$-action is a model for $E_{\mathcal F}G$. Therefore it suffices to show the map $LK^G_*(E_{\mathcal{F}}G\times E_{\mathcal F'}G)\longrightarrow LK^G_*(E_{\mathcal{F}'}G)$ induced by the projection is an isomorphism. Let $Z$ be any $G$-CW-complex with isotropy groups belonging to $\mathcal F'$, we show the map 
\begin{align}\label{iso}
LK^G_*(E_{\mathcal{F}}G\times Z)\longrightarrow LK^G_*(Z)
\end{align}
induced by the projection is an isomorphism and this will complete the proof.

If $Z$ is $0$-dimensional, then  
it is a disjoint union of $G$-spaces of the form $G/H$, where $H\in\mathcal F'$. Hence by additivity of $LK^G_*$, it suffices to show 
\begin{align}\label{iso1}
LK^G_*(E_{\mathcal{F}}G\times G/H)\longrightarrow LK^G_*(G/H)
\end{align}
is an isomorphism for any $H\in\mathcal F'$. By viewing $E_{\mathcal{F}}G$ as an $H$-space, we have $G$-homeomorphisms of $G$-CW-complexes
$$E_{\mathcal{F}}G\times G/H\cong G\times_H E_{\mathcal{F}}G,\ G/H\cong G\times_H\{pt\}.$$
Applying Theorem \ref{induction_str}, the map (\ref{iso1}) can be identified with the map
\begin{align}\label{iso2}
LK^H_*(E_{\mathcal{F}}G)\longrightarrow LK^H_*(\{pt\})
\end{align}
Notice that, by the fixed point set characterization of classifying spaces, $E_{\mathcal{F}}G$ is a model for $E_{H\cap\mathcal F}H$, hence by assumption, the map (\ref{iso2}) is an isomorphism. This proves the map (\ref{iso}) is an isomorphism when $Z$ is zero dimensional.

Assume the map (\ref{iso}) is an isomorphism for all $G$-CW complexes of dimension less than $d-1$, we prove it for $Z$ of dimension $d$. Let $Z_{d-1}$ be the $(d-1)$-th skeleton of $Z$. Consider a $G$-pushout 
$$
	\xymatrix{ 
			\ds\coprod_{i\in I}G/H_i\times S^{d-1}\ar[r]\ar[d] & Z_{d-1}\ar[d]\\ 
		    \ds\coprod_{i\in I}G/H_i\times D^d\ar[r] & Z
		}
	$$
Multiplying it by $E_{\mathcal F}G$, we obtain a $G$-pushout for $E_{\mathcal F}G\times Z$. The corresponding projection maps induce a map from the Mayer-Vietoris sequence for the $G$-pushout for $E_{\mathcal F}G\times Z$ to the Mayer-Vietoris sequence for the $G$-pushout for $Z$. By induction, it suffices to show the map
\begin{align}\label{iso3}
LK^G_*(E_{\mathcal{F}}G\times\coprod_{i\in I}G/H_i\times D^d)\longrightarrow LK^G_*(\coprod_{i\in I}G/H_i\times D^d)
\end{align}
is an isomorphism. But this follows from the additivity and $G$-homotopy invariance of $LK^G_*$, and the map (\ref{iso1}) is an isomorphism for any $H\in\mathcal F'$.

Finally, if $Z$ is infinite dimensional, then the map (\ref{iso}) is still an isomorphism by a colimit argument.
\end{proof}

\bibliographystyle{unsrt}
\bibliography{Topological_K-Homology_for_Non-proper_Actions_and_Assembly_Maps}
\end{document}